\documentclass[11pt]{article}
\usepackage{etex}
\reserveinserts{28}
\usepackage[all]{xy}

\usepackage{amsfonts}
\usepackage{amsthm}
\usepackage{enumerate}
\usepackage{graphicx}
\usepackage{mathrsfs}
\usepackage{bm}
\usepackage{cite}
\usepackage{amssymb,amsmath} % font used for R in Real numbers
\usepackage{makecell}
\usepackage{tikz}
\usepackage{pgf}
\usepackage{tikz}
\usetikzlibrary{patterns}
\usepackage{pgffor}
\usepackage{pgfcalendar}
\usepackage{pgfpages}
\usepackage{shuffle,yfonts}
\usepackage{mathtools}
\DeclareFontFamily{U}{shuffle}{}
\DeclareFontShape{U}{shuffle}{m}{n}{ <-8>shuffle7 <8->shuffle10}{}

\newcommand\Res{{\rm Res}}

\newcommand{\bfk}{{\boldsymbol{\sl{k}}}}

\newcommand{\bfx}{{\boldsymbol{\sl{x}}}}

 \allowdisplaybreaks
\usetikzlibrary{arrows,shapes,chains}
 \allowdisplaybreaks

\catcode`!=11
\let\!int\int \def\int{\displaystyle\!int}
\let\!lim\lim \def\lim{\displaystyle\!lim}
\let\!sum\sum \def\sum{\displaystyle\!sum}
\let\!sup\sup \def\sup{\displaystyle\!sup}
\let\!inf\inf \def\inf{\displaystyle\!inf}
\let\!cap\cap \def\cap{\displaystyle\!cap}
\let\!max\max \def\max{\displaystyle\!max}
\let\!min\min \def\min{\displaystyle\!min}
\let\!frac\frac \def\frac{\displaystyle\!frac}
\catcode`!=12

\let\oldsection\section
\renewcommand\section{\setcounter{equation}{0}\oldsection}

\allowdisplaybreaks

\DeclareMathOperator{\Li}{Li}

\def\N{\mathbb{N}}
\def\Z{\mathbb{Z}}

\def\ze{\zeta}

\theoremstyle{plain}
\newtheorem{thm}{Theorem}[section]
\newtheorem{lem}[thm]{Lemma}
\newtheorem{cor}[thm]{Corollary}

\theoremstyle{definition}

\newtheorem{re}[thm]{Remark}
\newtheorem{exa}[thm]{Example}

\begin{document}
%%%%%%%%%%%%%%%%%%%% title %%%%%%%%%%%%%%%%%%%%%%%%%%%%%%%%%%%%%%%%%%%%%%%%
\title{\bf Contour Integrations and Parity Results of Cyclotomic Euler Sums and Multiple Polylogarithm Function}
\author{
{Hongyuan Rui$^{a,}$\thanks{Email: rhy626514@163.com}\quad{and}\quad Ce Xu$^{b,}$\thanks{Email: cexu2020@ahnu.edu.cn}}\\[1mm]
a. \small School of Mathematics, Sichuan University,\\ \small Chengdu 610064, P.R. China\\
b. \small School of Mathematics and Statistics, Anhui Normal University,\\ \small Wuhu 241002, P.R. China
}

\date{}
\maketitle

\noindent{\bf Abstract.} In this paper, we define extended trigonometric functions via series and employ the method of contour integration to investigate the parity of certain cyclotomic Euler sums and multiple polylogarithm function. We can provide the statement of parity results for cyclotomic Euler sums of arbitrary order, explicit formulas for the parity of cyclotomic linear and quadratic Euler sums, as well as some formulas for the parity of cyclotomic cubic Euler sums and multiple polylogarithms. As a direct corollary, we derive known formulas concerning the parity of classical Euler sums and alternating Euler sums.

\medskip

\medskip


\section{Introduction}

In 1998, Flajolet and Salvy \cite{Flajolet-Salvy} investigated analytic formulas for a broad class of Dirichlet series known as ``\emph{Euler sums}" by considering contour integrals of the form
\begin{align*}
\oint_{(\infty)} r(s)\xi(s) ds:=\lim_{R\rightarrow \infty}\oint_{C_R} r(s)\xi(s) ds =0,
\end{align*}
where $C_R$ denote a circular contour with radius $R$. Here $\xi(s)$ is referred to as a kernel function, defined as
\begin{align*}
\xi(s)=\frac{\pi\cot(\pi s)\psi^{(p_1-1)}(-s)\psi^{(p_2-1)}(-s)\cdots \psi^{(p_k-1)}(-s)}{(p_1-1)!(p_2-1)!\cdots(p_k-1)!}
\end{align*} and $r(s)$ is a basis function, defined as $r(s)=1/s^q\ (\forall p_j\in\N,\ q\in \N\setminus \{0\})$. Here $\psi(s)$ denotes the the \emph{digamma function} defined by
\begin{align}\label{defn-classical-psi-funtion}
\psi(s)=-\gamma-\frac1{s}+\sum_{k=1}^\infty \left(\frac1{k}-\frac{1}{s+k}\right),
\end{align}
where $s\in\mathbb{C}\setminus \N_0^-$ and $\N_0^-:=\N^-\cup\{0\}=\{0,-1,-2,-3,\ldots\}$.

The \emph{classical Euler sums} they studied were defined as
\begin{align}\label{defn-classicalEulersums}
&{S_{{p_1,p_2,\ldots, p_k};q}} := \sum\limits_{n = 1}^\infty  {\frac{{H_n^{\left( {{p_1}} \right)}H_n^{\left( {{p_2}} \right)} \cdots H_n^{\left( {{p_k}} \right)}}}
{{{n^{q}}}}},
\end{align}
where $p_j\in \N$ and $q\geq 2$ with $1\leq p_1\leq p_2\leq \cdots \leq p_r$. When $k=1$, it is referred to as a \emph{linear Euler sum}, and when $k>1$, it is called a \emph{nonlinear Euler sum}. The quantity $p_1+\cdots+p_k+q$ is called the ``weight" of the sum, and the quantity $k$ is called the
``order". $H_n^{(p)}$ stands the {\emph{generalized harmonic number}} of order $p$ defined by
\[H_n^{(p)}:=\sum_{k=1}^n \frac{1}{k^p}\quad\text{and}\quad H_n\equiv H_n^{(1)}.\]
In particular, they proved the following theorem concerning the parity of classical Euler sums (\cite[Theorem 5.3]{Flajolet-Salvy} ): \emph{A Euler sum $S_{p_1,\ldots, p_k;q}$ with $k\geq 2$ reduces to a combination of sums of lower orders whenever the weight $p_1+\cdots+p_k+q$ and the order $k$ are of the same parity}. According to the series stuffle relations, the product of harmonic numbers $H_n^{p_1}\cdots H_n^{(p_k)}$  can be expressed as a linear combination of multiple harmonic sums. Consequently, every classical Euler sum can be represented as an $\Z$-coefficient linear combination of multiple zeta values. The explicit formulas were established by the second author of this paper and Wang in \cite{Xu-Wang2020}. For $k_1,\ldots,k_r$ are positive integers and $k_r\geq 2$, the \emph{multiple zeta values} (MZVs) are defined by (\cite{H1992,DZ1994})
\begin{align*}
\zeta(\bfk)\equiv \zeta(k_1,\ldots,k_r):=\sum_{0<n_1<\cdots<n_r} \frac{1}{n_1^{k_1}\cdots n_r^{k_r}},
\end{align*}
where $r$ and $k_1+\cdots+k_r$ are called the \emph{depth} and \emph{weight}, respectively. A systematic introduction to multiple zeta values and numerous research results on their properties up to 2016 can be found in Zhao's monograph \cite{Z2016}. Research on the parity of multiple zeta values has yielded abundant results. The study of this parity phenomenon originates from the work of Borwein and Girgensohn \cite{BG1996}. They postulated a conjecture which posits that if the weight $w$ and the depth $r$ of a multiple zeta value $\zeta(\bfk)$ are of opposite parity, then $\zeta(\bfk)$ can be expressed as a $\mathbb{Q}[\pi^2]$-linear combination of multiple zeta values with depth not exceeding $r-1$. This conjecture was subsequently proved and generalized by Bouillot\cite{Bouillot2014}, Ihara-Kaneko-Zagier \cite{IKZ2006}, Machide \cite{Machide2016} Panzer \cite{Panzer2017}, Tsumura \cite{Tsu-2004}, and among others. Regrettably, none of these parity results provide explicit formulas. Very recently, Hirose \cite{Hirose2025} derived an explicit formula for the parity of multiple zeta values by employing the theory of multitangent functions developed by Bouillot \cite{Bouillot2014}.

The second author of this paper and Wang employed similar contour integration techniques to establish parity results for Euler sums involving odd harmonic numbers (see \cite[Theorem 5.1]{Xu-Wang2022}). From these results, explicit formulas were derived for the parity of Hoffman's multiple $t$-values (\cite{H2019}) and Kaneko-Tsumura's multiple $T$-values (\cite{KanekoTs2019}) at depths $\leq 3$.
For $\bfk=(k_1,\ldots,k_r)\in\N^r$ with $k_r>1$, the \emph{multiple $t$-values} and \emph{multiple $T$-values} are defined by
\begin{align*}
t(\bfk):=\sum_{\substack{0<n_1<\cdots<n_r\\ n_j: \text{ odd}}}\frac{2^r}{n_1^{k_1}\cdots n_r^{k_r}}\quad\text{and}\quad T(\bfk):=\sum_{\substack{0<n_1<\cdots<n_r\\ n_j\equiv j\pmod{2}}} \frac{2^r}{n_1^{k_1}\cdots n_r^{k_r}}.
\end{align*}
In fact, Zhao \cite{Zhao2015} had begun studying some sum formulas for multiple $t$-values a few years prior to Hoffman's formal definition of multiple $t$-values. Recent studies on sum formulas for multiple $t$-values and multiple $T$-values can be found in Zhao's paper \cite{Z2024} and the references therein.

In this paper, we begin by defining extended trigonometric functions through a generalized digamma function. Subsequently, by considering contour integrals involving both the generalized digamma function and the extended trigonometric functions, we investigate the parity of \emph{generalized Euler sums} of the following form:
\begin{align}\label{defn-GeneralizedEulerSums}
S_{p_1,\ldots, p_k;q}(x_1,\ldots,x_k;x):=\sum_{n=1}^\infty \frac{\zeta_{n}(p_1;x_1)\zeta_{n}(p_2;x_2)\cdots \zeta_{n}(p_k;x_k)}{n^q}x^n,
\end{align}
where $p_1,\ldots,p_k,q\in \N$ and $|x_1\cdots x_kx|\leq 1$ with $1\leq p_1\leq p_2\leq \cdots \leq p_r$ and $(q,x)\neq (1,1)$. When $x_1,\ldots,x_k,x$ are all roots of unity, then we call them the \emph{cyclotomic Euler sums}. If $x_1,\ldots,x_r\in\{\pm 1\}$ and at least one $x_j=-1$, then they are called \emph{alternating Euler sums}. Similar to the classical Euler sum \eqref{defn-classicalEulersums}, the quantity $p_1+\cdots+p_k+q$ is called the ``weight" of the sum, and the quantity $k$ is called the
``order".  Here $\zeta_n(p;x)$ stands the finite sum of polylogarithm function defined by
\begin{align}
\zeta_n(p;x):=\sum_{k=1}^n \frac{x^k}{k^p}\quad (p\in \N,\ |x|\leq 1),
\end{align}
and the \emph{polylogarithm function} $\Li_p(x)$ is defined by
\begin{align}
\Li_{p}(x):=\lim_{n\rightarrow \infty}\zeta_n(p;x)= \sum_{n=1}^\infty \frac{x^n}{n^p}\quad (|x|\leq 1,\ (p,x)\neq (1,1),\ p\in \N).
\end{align}
For any $\bfk=(k_1,\dotsc,k_r)\in\N^r$ and $\bfx=(x_1,\ldots,x_r)\in \mathbb{C}^r$, the classical \emph{multiple polylogarithm function} with $r$-variables is defined by
\begin{align}\label{defn-mpolyf}
\Li_{\bfk}(\bfx)\equiv \Li_{k_1,\dotsc,k_r}(x_1,\dotsc,x_r):=\sum_{0<n_1<\cdots<n_r} \frac{x_1^{n_1}\dotsm x_r^{n_r}}{n_1^{k_1}\dotsm n_r^{k_r}}
\end{align}
which converges if $|x_j\cdots x_r|<1$ for all $j=1,\dotsc,r$. It can be analytically continued to a multi-valued meromorphic function on $\mathbb{C}^r$ (see \cite{Zhao2007d}). In particular, if $(k_1,\dotsc, k_r)\in\N^r$ and $x_1,\ldots,x_r$ are $N$th roots of unity, then \eqref{defn-mpolyf} become the \emph{cyclotomic (or colored) multiple zeta values of level $N$} which converges if $(k_r,x_r)\ne (1,1)$ (see \cite{YuanZh2014a} and \cite[Ch. 15]{Z2016}). The research results on cyclotomic multiple zeta values are also very extensive. For some recent research work in this area, please refer to literature \cite{Anzawa2026,BTT2018,Li2024,JLi2025,SZ2020,Tasaka2021} and references therein. The level two cyclotomic multiple zeta values are called \emph{alternating multiple zeta values}. In this case, namely,
when $(x_1,\dotsc,x_r)\in\{\pm 1\}^r$ and $(k_r,x_r)\ne (1,1)$, we set
$\zeta(\bfk;\bfx)= \Li_\bfk(\bfx)$. Further, we put a bar on top of
$k_j$ if $x_j=-1$. For example,
\begin{equation*}
\zeta(\bar 3,2,\bar 1,4)=\zeta(3,2,1,4;-1,1,-1,1).
\end{equation*}
Both the multiple $t$-values defined by Hoffman (\cite{H2019}) and the multiple $T$-values defined by Kaneko and Tsumura (\cite{KanekoTs2019}) can be regarded as a type of level two multiple zeta values, as they can both be expressed as $\Z$-linear combinations of alternating multiple zeta values.
It is evident from the \emph{stuffle relations} (see \cite{H2000}) that every generalized Euler sum can be expressed as an $\Z$-coefficient linear combination of multiple polylogarithm functions. For example: for generalized linear and quadratic Euler sums, we have
\begin{align}
&S_{p;q}(x;y)=\Li_{p,q}(x,y)+\Li_{p+q}(xy),\label{doubleCESDPL}\\
&S_{p_1,p_2;q}(x_1,x_2;x)=\Li_{p_1,p_2,q}(x_1,x_2,x)+\Li_{p_2,p_1,q}(x_2,x_1,x)+\Li_{p_1+p_2,q}(x_1x_2,x)\nonumber
\\&\qquad\qquad\qquad\qquad+\Li_{p_1,p_2+q}(x_1,x_2x)+\Li_{p_2,p_1+q}(x_2,x_1x)+\Li_{p_1+p_2+q}(x_1x_2x).\label{TripleCESTPL}
\end{align}
Clearly, when $x=1$ and all $x_j=1$, the generalized Euler sum \eqref{defn-GeneralizedEulerSums} reduces to the classical Euler sum \eqref{defn-classicalEulersums}.

The objective of this paper is to employ the contour integral method to establish parity results for both generalized and cyclotomic Euler sums. In particular, we can provide explicit formulas for the parity of cyclotomic linear and quadratic Euler sums. Furthermore, by leveraging the relationships between generalized Euler sums and multiple polylogarithms, as well as between cyclotomic Euler sums and cyclotomic multiple zeta values, we also derive certain parity results for multiple polylogarithms and cyclotomic multiple zeta values. One of the primary results of this paper is the proof of the following theorem regarding the parity of cyclotomic Euler sums (see Theorem \ref{thm-parityc-C-ES-one}):
\begin{thm}\label{thm-parityc-C-ESmain} Let $x,x_1,\ldots,x_r$ be roots of unity, and $p_1,\ldots,p_r,q\geq 1$ with $(p_j,x_j)\ (j=1,2,\ldots,r) $ and $ (q,x)\neq (1,1)$. We have
\begin{align*}
(-1)^r S_{p_1,p_2,\ldots,p_r;q}\Big(x_1,x_2,\ldots,x_r;x\Big)+(-1)^{p_1+\cdots+p_r+q}S_{p_1,p_2,\ldots,p_r;q}\Big(x_1^{-1},x_2^{-1},\ldots,x_r^{-1};x^{-1}\Big)
\end{align*}
reduces to a combination of cyclotomic Euler sums of lower orders.
\end{thm}
The above-mentioned parity theorem concerning cyclotomic Euler sums corresponds to the parity theorem for multiple polylogarithms proved by Panzer \cite[Thm 1.3]{Panzer2017}, which states that: for all $r\in \N$ and $\bfk=(k_1,\ldots,k_r)\in \N^r$, the function
\begin{align*}
\Li_{\bfk}(z_1,z_2,\ldots,z_r)-(-1)^{k_1+\cdots+k_r+r}\Li_{\bfk}(1/z_1,1/z_2,\ldots,1/z_r)
\end{align*}
is of depth at most $r-1$. Here $(z_1,\ldots,z_r)\in \mathbb{C}^r \setminus \bigcup_{1\leq i\leq j\leq r}\{(z_1,\ldots,z_r): z_iz_{i+1}\cdots z_j\in[0,+\infty)\}$.

\section{Preliminaries}

Define the \emph{generalized digamma function} $\phi(s;x)$ by
\begin{align}
\phi(s;x):=\sum_{k=0}^\infty \frac{x^k}{k+s}\quad (s\notin\N_0^-:=\{0,-1,-2,-3,\ldots\}),
\end{align}
where $x$ is an arbitrary complex number with $|x|\leq 1$ and $x\neq 1$.

For the subsequent contour integration and residue calculations, we need to derive either the Laurent series expansion or Taylor series expansion of this function at integer points.

By direct calculations, we obtain that if $|s+n|<1\ (n\in \N_0:=\{0,1,2,3,\ldots\})$, then
\begin{align}\label{Lexp-phi--n}
\phi(s;x)&=\frac{x^n}{s+n}+\sum_{k=0,\atop k\neq n}^\infty \frac{x^k}{k-n+s+n}=\frac{x^n}{s+n}+\sum_{k=0,\atop k\neq n}^\infty \frac{x^k}{(k-n)\left(1+\frac{s+n}{k-n} \right)}\nonumber\\
&=\frac{x^n}{s+n}+\sum_{k=0,\atop k\neq n}^\infty \frac{x^k}{k-n}\sum_{m=0}^\infty \left(-\frac{s+n}{k-n}\right)^m=\frac{x^n}{s+n}+\sum_{m=0}^\infty (-1)^m (s+n)^m \sum_{k=0,\atop k\neq n}^\infty \frac{x^k}{(k-n)^{m+1}} \nonumber\\
&=\frac{x^n}{s+n}+\sum_{m=0}^\infty \left((-1)^m\Li_{m+1}(x)-\zeta_n\Big(m+1;x^{-1}\Big) \right)x^n(s+n)^m,
\end{align}
and if $|s-n|<1\ (n\in \N)$, then
\begin{align}\label{Lexp-phi-n}
\phi(s;x)&=\sum_{k=0}^\infty \frac{x^k}{k+n+s-n}=\sum_{k=0}^\infty  \frac{x^k}{k+n} \frac{1}{1+\frac{s-n}{k+n}} \nonumber\\
&=\sum_{k=0}^\infty  \frac{x^k}{k+n} \sum_{m=0}^\infty \left(-\frac{s-n}{k+n}\right)^m=\sum_{m=0}^\infty (-1)^m (s-n)^m \sum_{k=0}^\infty  \frac{x^k}{(k+n)^{m+1}}  \nonumber\\
&=\sum_{m=0}^\infty (-1)^m\left(\Li_{m+1}(x)-\zeta_{n-1}(m+1;x) \right)x^{-n}(s-n)^m.
\end{align}

Taking the $(p-1)$th-order derivative of \eqref{Lexp-phi--n} and \eqref{Lexp-phi-n} with respect to $s$ respectively, we obtain
\begin{align}\label{Lexp-phi--n-diffp-1}
\frac{\phi^{(p-1)}(s;x)}{(p-1)!} (-1)^{p-1}&=x^n\sum_{k=0}^\infty \binom{k+p-1}{p-1} \left((-1)^k \Li_{k+p}(x)+(-1)^p\zeta_n\Big(k+p;x^{-1}\Big)\right)(s+n)^k \nonumber\\&\quad+\frac{x^n}{(s+n)^p}\qquad (|s+n|<1,\ n\geq 0)
\end{align}
and
\begin{align}\label{Lexp-phi-n-diffp-1}
\frac{\phi^{(p-1)}(s;x)}{(p-1)!} (-1)^{p-1}&=x^{-n}\sum_{k=0}^\infty \binom{k+p-1}{p-1} (-1)^k  \left( \Li_{k+p}(x)-\zeta_{n-1}\Big(k+p;x\Big)\right)(s-n)^k \nonumber\\&\qquad\qquad\qquad\qquad (|s-n|<1,\ n\geq 1).
\end{align}

Now, we define the \emph{extended trigonometric function} $\Phi(s;x)$ by
\[\Phi(s;x):=\phi(s;x)-\phi\Big(-s;x^{-1}\Big)-\frac1{s},\]
where to ensure the convergence of the series above, $x$ can only be any root of unity.  Clearly, all integers are simple poles of this function. In particular, $\Phi(s;1)=\pi \cot(\pi s)$ if $x=1$ and $\Phi(s;-1)=\pi \csc(\pi s)$ if $x=-1$. Applying \eqref{Lexp-phi--n} and \eqref{Lexp-phi-n}, we deduce that for any $s\rightarrow n\ (n\in \Z)$,
\begin{align}\label{LEPhi-function}
\Phi(s;x)=x^{-n} \left(\frac1{s-n}+\sum_{m=0}^\infty \Big((-1)^m\Li_{m+1}(x)-\Li_{m+1}\Big(x^{-1}\Big)\Big)(s-n)^m \right).
\end{align}
Setting $x=\pm 1$ then we have (see \cite{Flajolet-Salvy})
\begin{align}
&\pi \cot \left( {\pi s} \right)\mathop  = \limits^{s \to n} \frac{1}{{s - n}} - 2\sum\limits_{k = 1}^\infty  {\zeta( 2k){{\left( {s - n} \right)}^{2k - 1}}},\label{expansion-one-sine}\\
&\frac{\pi }
{{\sin \left( {\pi s} \right)}}\mathop  = \limits^{s \to n} {\left( { - 1} \right)^n}\left( {\frac{1}
{{s - n}} + 2\sum\limits_{k = 1}^\infty  {\bar \zeta ( {2k} ){{\left( {s - n} \right)}^{2k - 1}}} } \right),\label{expansion-one-cosine}
\end{align}
where $\zeta(s)$ and ${\bar \zeta} (s)$ denote the \emph{Riemann zeta function} and \emph{alternating Riemann zeta function} defined by
\[\zeta(s):=\sum_{n=1}^\infty \frac{1}{n^s}\quad (\Re(s)>1)\quad\text{and}\quad\bar \zeta ( s ) := \sum\limits_{n = 1}^\infty  {\frac{{{{\left( { - 1} \right)}^{n - 1}}}}{{{n^s}}}} \quad(\Re(s)>0).\]

Flajolet and Salvy \cite{Flajolet-Salvy} defined a kernel function $\xi(s)$ with two requirements: 1). $\xi(s)$ is meromorphic in the whole complex plane. 2). $\xi(s)$ satisfies $\xi(s)=o(s)$ over an infinite collection of circles $\left| s \right| = {\rho _k}$ with ${\rho _k} \to \infty $. Applying these two conditions of kernel
function $\xi(s)$, Flajolet and Salvy discovered the following residue lemma.

\begin{lem}\emph{(cf.\ \cite{Flajolet-Salvy})}\label{lem-redisue-thm}
Let $\xi(s)$ be a kernel function and let $r(s)$ be a rational function which is $O(s^{-2})$ at infinity. Then
\begin{align}\label{residue-}
\sum\limits_{\alpha  \in O} {{\mathop{\rm Res}}{{\left( {r(s)\xi(s)},\alpha  \right)}}}  + \sum\limits_{\beta  \in S}  {{\mathop{\rm Res}}{{\left( {r(s)\xi(s)}, \beta  \right)}}}  = 0,
\end{align}
where $S$ is the set of poles of $r(s)$ and $O$ is the set of poles of $\xi(s)$ that are not poles $r(s)$. Here ${\mathop{\rm Re}\nolimits} s{\left( {r(s)},\alpha \right)} $ denotes the residue of $r(s)$ at $s= \alpha.$
\end{lem}

Clearly, on the circle with radius $n+1/2\ (s\in \N)$, the functions $\phi(s;x)$, $\Phi(s;x)$ and their derivatives are all $O(|s|^\varepsilon)\ (\forall \varepsilon>0)$. Consequently, any
polynomial form in $\Phi(s;x)$ and $\phi^{(j)}(s;x)$ is itself a kernel function with poles at a subset of the integers. In this paper, we investigate the parity of cyclotomic Euler sums and multiple polylogarithms primarily by considering contour integrals of the form for the following two types: for $p_1,\ldots,p_r,q\in \N$,
\begin{align*}
\oint\limits_{\left( \infty  \right)}F_{p_1p_2\cdots p_r,q}(s)ds:=\oint\limits_{\left( \infty  \right)} \frac{\Phi(s;x)\phi^{(p_1-1)}(s;x_1)\cdots\phi^{(p_r-1)}(s;x_r)}{(p_1-1)!\cdots (p_r-1)!s^q}(-1)^{p_1+\cdots+p_r-r} ds=0
\end{align*}
and
\begin{align*}
\oint\limits_{\left( \infty  \right)}G_{p_1p_2\cdots p_r,q}(s)ds:=\oint\limits_{\left( \infty  \right)} \frac{\phi^{(p_1-1)}(s;x_1)\cdots\phi^{(p_r-1)}(s;x_r)}{(p_1-1)!\cdots (p_r-1)!s^q}(-1)^{p_1+\cdots+p_r-r} ds=0.
\end{align*}

\section{Parity Results for Cyclotomic Linear Euler Sums}
In this section, we examine the parity of cyclotomic linear Euler sums and cyclotomic double zeta values, as well as provide some formulas for the double polylogarithms.
First, we consider the contour integration (Here $p,q\in \N$)
\begin{align*}
\oint\limits_{\left( \infty  \right)} \frac{\Phi(s;x)\phi^{(p-1)}(s;y)}{(p-1)!s^q} (-1)^{p-1}ds=0.
\end{align*}

\begin{thm}\label{thm-double-CES} Let $x,y$ be roots of unity, and $p,q\geq 1$ with $(p,y), (q,xy)\neq (1,1)$. We have
\begin{align}
&\Li_{p,q}\left(y,\frac1{xy}\right)-(-1)^{p+q}\Li_{p,q}\left(\frac1{y},xy\right)\nonumber\\
&=\Li_p(y)\Li_q\left(\frac1{xy}\right)+(-1)^q \Li_p(y)\Li_q(xy)-\Li_{p+q}\left(\frac1{x}\right)\nonumber\\
&\quad-(-1)^q \sum_{l=0}^p \binom{p+q-l-1}{q-1} \left((-1)^l\Li_l(x)+\Li_l\left(\frac1{x}\right)\right)\Li_{p+q-l}(xy)\nonumber\\
&\quad-(-1)^q\sum_{l=0}^q \binom{p+q-l-1}{p-1} \left((-1)^l \Li_l\left(\frac1{x}\right)+\Li_l(x)\right)\Li_{p+q-l}(y),
\end{align}
where $\Li_0(x)+\Li_0(x^{-1}):=-1$.
\end{thm}
\begin{proof}
Letting
\[F_{p,q}(s):= \frac{\Phi(s;x)\phi^{(p-1)}(s;y)}{(p-1)!s^q} (-1)^{p-1}.\]
Clearly, the function $F_{p,q}(s)$ only singularities are poles at the integers. At a positive integer $n\in \N$, the pole $s=n$ is simple and by the expansions \eqref{Lexp-phi-n-diffp-1} and \eqref{LEPhi-function}, the residue is
\begin{align*}
\Res\left(F_{p,q}(s),n\right)=\frac{(xy)^{-n}}{n^q} \Big(\Li_p(y)-\zeta_{n-1}(p;y)\Big).
\end{align*}
For positive integer $n$, the pole $s=-n$ has order $p+1$. By \eqref{Lexp-phi--n-diffp-1} and \eqref{LEPhi-function}, the residue is
\begin{align*}
\Res\left(F_{p,q}(s),-n\right)&=\frac1{p!} \lim_{s\rightarrow -n} \frac{d^p}{ds^p}\left((s+n)^{p+1} F_{p,q}(s)\right)\\
&=(-1)^q \binom{p+q-1}{p} \frac{(xy)^n}{n^{p+q}}+(-1)^q \frac{\Li_p(y)+(-1)^p\zeta_n\Big(p;y^{-1}\Big)}{n^q}(xy)^n\\
&\quad+(-1)^q \sum_{m=0}^{p-1} \binom{p+q-m-2}{q-1}\left((-1)^m\Li_{m+1}(x)-\Li_{m+1}\Big(x^{-1}\Big) \right) \frac{(xy)^n}{n^{p+q-m-1}}.
\end{align*}
The pole $s=0$ has order $p+q+1$, the residue is
\begin{align*}
\Res\left(F_{p,q}(s),0\right)&=(-1)^{p+q-1}\Li_{p+q}(x)-\Li_{p+q}\Big(x^{-1}\Big)+(-1)^q\binom{p+q-1}{p-1}\Li_{p+q}(y)\\
&\quad+\sum_{m+k=q-1,\atop m,k\geq 0} (-1)^k\binom{k+p-1}{p-1}\Li_{k+p}(y)\left((-1)^m \Li_{m+1}(x)-\Li_{m+1}\Big(x^{-1}\Big)\right).
\end{align*}
By Lemma \ref{lem-redisue-thm}, we know that
\[\sum_{n=1}^\infty \left(\Res\left(F_{p,q}(s),n\right)+\Res\left(F_{p,q}(s),-n\right)\right)+\Res\left(F_{p,q}(s),0\right)=0.\]
Finally, combining these three contributions yields the statement of Theorem \ref{thm-double-CES}.
\end{proof}

From equation \eqref{doubleCESDPL}, the parity formula for cyclotomic linear Euler sums $S_{p;q}(x;y)$ can be obtained.

If we consider contour integrals involving only the function $\phi(s;x)$ and rational functions, we can still obtain some results on generalized Euler sums, though the expressions become more complicated. The following theorem provides one such conclusion.

\begin{thm}\label{thm-double-CES-Comb}
 For positive integers $p_1,p_2,q$ and $x_1,x_2$ are arbitrary complex numbers with $0<|x_1|, |x_2|\leq 1$ and $x_1,x_2 \neq 1$, we have
\begin{align}
&(-1)^{p_1}\sum_{k=0}^{p_2-1}\binom{k+p_1-1}{p_1-1}\binom{q+p_2-k-2}{q-1} S_{k+p_1;q+p_2-k-1}\Big(x_1^{-1};x_1x_2\Big)\nonumber\\
&+(-1)^{p_2}\sum_{k=0}^{p_1-1}\binom{k+p_2-1}{p_2-1}\binom{q+p_1-k-2}{q-1} S_{k+p_2;q+p_1-k-1}\Big(x_2^{-1};x_1x_2\Big)\nonumber\\
&=(-1)^{p_1}\binom{p_1+p_2+q-2}{p_2-1}\Li_{p_1+p_2+q-1}(x_2)+(-1)^{p_2}\binom{p_1+p_2+q-2}{p_1-1}\Li_{p_1+p_2+q-1}(x_1)\nonumber\\
&\quad+\sum_{k_1+k_2=q-1,\atop k_1,k_2\geq 0} \binom{k+p_1-1}{p_1-1}\binom{k+p_2-1}{p_2-1} \Li_{k_1+p_1}(x_1)\Li_{k_2+p_2}(x_2)\nonumber\\
&\quad-\sum_{k=0}^{p_2-1}\binom{k+p_1-1}{p_1-1}\binom{q+p_2-k-2}{q-1}(-1)^k \Li_{k+p_1}(x_1)\Li_{q+p_2-k-1}(x_1x_2)\nonumber\\
&\quad-\sum_{k=0}^{p_1-1}\binom{k+p_2-1}{p_2-1}\binom{q+p_1-k-2}{q-1}(-1)^k \Li_{k+p_2}(x_1)\Li_{q+p_1-k-1}(x_1x_2)\nonumber\\
&\quad-\binom{p_1+p_2+q-2}{q-1} \Li_{p_1+p_2+q-1}(x_1x_2).
\end{align}
\end{thm}
\begin{proof}
Letting
\[\oint\limits_{\left( \infty  \right)}G_{p_1p_2,q}(s)ds:= \oint\limits_{\left( \infty  \right)}\frac{\phi^{(p_1-1)}(s;x_1)\phi^{(p_2-1)}(s;x_2)}{(p_1-1)!(p_2-1)!s^q} (-1)^{p_1+p_2}ds=0.\]
Clearly, the function $G_{p_1p_2,q}(s)$ only singularities are poles at the non-positive integers. At a positive integer $n\in \N$, the pole $s=-n$ is a pole of order $p_1+p_2$. Applying \eqref{Lexp-phi--n-diffp-1}, we have
\begin{align*}
G_{p_1p_2,q}(s)&=\frac{(x_1x_2)^n}{(s+n)^{p_1+p_2}}\\&
\quad+(x_1x_2)^n\sum_{k=0}^{p_2-1}\binom{k+p_1-1}{p_1-1}\left((-1)^k\Li_{k+p_1}(x_1)+(-1)^{p_1}\zeta_n\Big(k+p_1;x_1^{-1}\Big) \right)(s+n)^{k-p_2}
\\&
\quad+(x_1x_2)^n\sum_{k=0}^{p_1-1}\binom{k+p_2-1}{p_2-1}\left((-1)^k\Li_{k+p_2}(x_2)+(-1)^{p_2}\zeta_n\Big(k+p_2;x_2^{-1}\Big) \right)(s+n)^{k-p_1}\\
&\quad+\cdots.
\end{align*}
By a direct residue computation, one obtains
\begin{align*}
\Res\left(G_{p_1p_2,q}(s),-n\right)&=\frac1{(p_1+p_2-1)!} \lim_{s\rightarrow -n} \frac{d^{p_1+p_2-1}}{ds^{p_1+p_2-1}}\left((s+n)^{p_1+p_2} G_{p_1p_2,q}(s)\right)\\
&=(-1)^q\sum_{k=0}^{p_2-1}\binom{k+p_1-1}{p_1-1}\binom{q+p_2-k-2}{q-1}\\&\qquad\qquad\times\left((-1)^k\Li_{k+p_1}(x_1)+(-1)^{p_1}\zeta_n\Big(k+p_1;x_1^{-1}\Big)\right)\frac{(x_1x_2)^n}{n^{q+p_2-k-1}}\\
&\quad+(-1)^q\sum_{k=0}^{p_1-1}\binom{k+p_2-1}{p_2-1}\binom{q+p_1-k-2}{q-1}\\&\qquad\qquad\times\left((-1)^k\Li_{k+p_2}(x_2)+(-1)^{p_2}\zeta_n\Big(k+p_2;x_2^{-1}\Big)\right)\frac{(x_1x_2)^n}{n^{q+p_1-k-1}}\\
&\quad+(-1)^q \binom{q+p_1+p_2-2}{q-1}\frac{(x_1x_2)^n}{n^{q+p_1+p_2-1}}.
\end{align*}
The pole $s=0$ has order $p_1+p_2+q$, the residue is
\begin{align*}
\Res\left(G_{p_1p_2,q}(s),0\right)&=\frac1{(p_1+p_2+q-1)!} \lim_{s\rightarrow 0} \frac{d^{p_1+p_2+q-1}}{ds^{p_1+p_2+q-1}}\left(s^{p_1+p_2} G_{p_1p_2,q}(s)\right)\\
&=(-1)^{p_1+q-1}\binom{p_1+p_2+q-2}{p_2-1}\Li_{p_1+p_2+q-1}(x_2)\\&\quad+(-1)^{p_2+q-1}\binom{p_1+p_2+q-2}{p_1-1}\Li_{p_1+p_2+q-1}(x_1)\\
&\quad+(-1)^{q-1} \sum_{k_1+k_2=q-1,\atop k_1,k_2\geq 0} \binom{k+p_1-1}{p_1-1}\binom{k+p_2-1}{p_2-1} \Li_{k_1+p_1}(x_1)\Li_{k_2+p_2}(x_2).
\end{align*}
Therefore, using Lemma \ref{lem-redisue-thm} and combining these two contributions yields the statement of Theorem \ref{thm-double-CES-Comb}.
\end{proof}

Setting $p_1=p_2=1$ in Theorem \ref{thm-double-CES-Comb} gives the following corollary.

\begin{cor} For positive integer $q$ and $x_1,x_2$ are arbitrary complex numbers with $0<|x_1|, |x_2|\leq 1$ and $x_1,x_2 \neq 1$, we have
\begin{align}
&S_{1;q}\Big(x_1^{-1};x_1x_2\Big)+S_{1;q}\Big(x_2^{-1};x_1x_2\Big)\nonumber\\
&=q\Li_{q+1}(x_1x_2)+\Li_{q+1}(x_1)+\Li_{q+1}(x_2)+\Big(\Li_1(x_1)+\Li_1(x_2)\Big)\Li_q(x_1x_2)\nonumber\\&\quad-\sum_{k_1+k_2=q-1,\atop k_1,k_2\geq 0} \Li_{k_1+1}(x_1)\Li_{k_2+1}(x_2).
\end{align}
\end{cor}
In particular, if letting $x_1=x_2=x$ yields
\begin{align}\label{equ-section2-classicalEs}
S_{1;q}\Big(x^{-1};x^2\Big)&=\frac{q}{2}\Li_{q+1}(x^2)+\Li_{q+1}(x)+\Li_1(x)\Li_q(x^2)\nonumber\\&\quad-\frac1{2}\sum_{k_1+k_2=q-1,\atop k_1,k_2\geq 0} \Li_{k_1+1}(x)\Li_{k_2+1}(x).
\end{align}
Clearly, when $q>1$ in Theorem \ref{thm-double-CES-Comb}, the generalized Euler sum on the left-hand side remains convergent as $x_1,x_2$ approach $1$, thereby reducing to the classical Euler sum. Consequently, the divergent terms that would otherwise arise in the limit of the right-hand series will cancel each other out. For example, setting $x\rightarrow 1$ in \eqref{equ-section2-classicalEs} yields the well-known result (see \cite[Thm. 2.2]{Flajolet-Salvy})
\begin{align}
S_{1;q}=\sum_{n=1}^\infty \frac{H_n}{n^q}=\left(1+\frac{q}{2}\right)\ze(q+1)-\frac1{2}\sum_{k=1}^{q-2}\ze(k+1)\ze(q-k).
\end{align}

\section{Parity Results for Cyclotomic Quadratic Euler Sums}
In this section, we will employ the method of contour integration, via residue computation, to investigate explicit formulas for the parity of cyclotomic quadratic Euler sums and present some results for generalized quadratic Euler sums. This will further lead to certain parity results for cyclotomic multiple zeta values of depth three. First, we consider the parity results for cyclotomic quadratic Euler sums. To better present the results in the subsequent discussion, we adopt the following notation unless otherwise specified:
\begin{align}
\sum_{\sigma\in \left\{\Big(a_1^{(1)}a_2^{(1)}\cdots a_r^{(1)}\Big),\ldots, \Big(a_1^{(k)}a_2^{(k)}\cdots a_r^{(k)}\Big)\right\}} f\Big(x_{\sigma(1)},\ldots,x_{\sigma(r)}\Big):=\sum_{j=1}^k f\Big(x_{a_1^{(j)}},\ldots,x_{a_r^{(j)}}\Big),
\end{align}
where $a_1^{(j)},\ldots,a_r^{(j)}\in \N$ and $j=1,2,\ldots,k\ (k\in \N)$. The $\Big(a_1^{(j)}a_2^{(j)}\cdots a_r^{(j)}\Big)$ is a permutation of distinct positive integers, and the order of the terms $a_i^{(j)}$ in the sequence $\Big(a_1^{(j)}a_2^{(j)}\cdots a_r^{(j)}\Big)$  is significant (i.e., their positions are not interchangeable). For example, $\Big(a_1^{(j)}a_2^{(j)}\cdots a_r^{(j)}\Big)$ and $\Big(a_2^{(j)}a_3^{(j)}\cdots a_r^{(j)}a_1^{(j)}\Big)$  represent distinct objects.

\begin{thm}\label{thm-quadratic-C-ES-one} Let $x,x_1,x_2$ be roots of unity, and $p_1,p_2,q\geq 1$ with $(p_1,x_1), (p_2,x_2) $ and $ (q,xx_1x_2)\neq (1,1)$. We have
\begin{align}
&S_{p_1,p_2;q}\left(x_1,x_2;\left(xx_1x_2\right)^{-1}\right)+\left(-1\right)^{p_1+p_2+q}S_{p_1,p_2;q}\left(x_1^{-1},x_2^{-1};xx_1x_2\right)\nonumber\\
&=-\Li_{p_1}\left(x_1\right)\Li_{p_2}\left(x_2\right)\Li_{q}\left(\left(xx_1x_2\right)^{-1}\right)-\Li_{p_1+p_2+q}\left(x^{-1}\right)\nonumber\\
&\quad+\sum_{\sigma\in\left\{\left(12\right),\left(21\right)\right\}}S_{p_{\sigma\left(1\right)};p_{\sigma\left(2\right)}+q}\left(x_{\sigma\left(1\right)};x_{\sigma\left(2\right)}\left(xx_1x_2\right)^{-1}\right)-\left(-1\right)^q\binom{p_1+p_2+q-1}{p_1+p_2}\Li_{p_1+p_2+q}\left(xx_1x_2\right)\nonumber\\
&\quad+\sum_{\sigma\in\left\{\left(12\right),\left(21\right)\right\}}\Li_{p_{\sigma\left(1\right)}}\left(x_{\sigma\left(1\right)}\right)
\left(S_{p_{\sigma\left(2\right)};q}\left(x_{\sigma\left(2\right)};\left(xx_1x_2\right)^{-1}\right)-\Li_{p_{\sigma\left(2\right)}+q}\left(x_{\sigma\left(2\right)}\left(xx_1x_2\right)^{-1}\right)\right)\nonumber\\
&\quad-\sum_{\sigma\in\left\{\left(12\right),\left(21\right)\right\}}\sum_{k=0}^{p_{\sigma\left(2\right)}}\binom{k+p_{\sigma\left(1\right)}-1}
{p_{\sigma\left(1\right)}-1}\binom{q+p_{\sigma\left(2\right)}-k-1}{q-1}\nonumber\\
&\qquad\times\left(\left(-1\right)^{k+q}\Li_{k+p_{\sigma\left(1\right)}}\left(x_{\sigma\left(1\right)}\right)\Li_{p_{\sigma\left(2\right)}+q-k}\left(xx_1x_2\right)
+\left(-1\right)^{p_{\sigma\left(1\right)}+q}S_{k+p_{\sigma\left(1\right)};p_{\sigma\left(2\right)}+q-k}\left(x_{\sigma\left(1\right)}^{-1};xx_1x_2\right)\right)\nonumber\\
&\quad-\sum_{\sigma\in\left\{\left(21\right),\left(12\right)\right\}}\sum_{0\le k_1+k_2\le p_{\sigma\left(2\right)}-1}\left(\left(-1\right)^{k_1}\Li_{k_1+1}\left(x\right)-\Li_{k_1+1}\left(x^{-1}\right)\right)\binom{k_2+p_{\sigma\left(1\right)}-1}{p_{\sigma\left(1\right)}-1}\nonumber\\
&\qquad\times\binom{q+p_{\sigma\left(2\right)}-k_1-k_2-2}{q-1}\begin{Bmatrix}
	\left(-1\right)^{k_2+q}\Li_{k_2+p_{\sigma\left(1\right)}}\left(x_{\sigma\left(1\right)}\right)\Li_{p_{\sigma\left(2\right)}+q-k_1-k_2-1}\left(xx_1x_2\right)\\
	+\left(-1\right)^{p_{\sigma\left(1\right)}+q}S_{k_2+p_{\sigma\left(1\right)};p_{\sigma\left(2\right)}+q-k_1-k_2-1}\left(x_{\sigma\left(1\right)}^{-1};xx_1x_2\right)
\end{Bmatrix}\nonumber\\
&\quad-\left(-1\right)^q\Li_{p_1}\left(x_1\right)\Li_{p_2}\left(x_2\right)\Li_{q}\left(xx_1x_2\right)\nonumber\\
&\quad-\sum_{\sigma\in\left\{\left(21\right),\left(12\right)\right\}}\left(-1\right)^{p_{\sigma\left(1\right)}+q}\Li_{p_{\sigma\left(2\right)}}\left(x_{\sigma\left(2\right)}\right)
S_{p_{\sigma\left(1\right)};q}\left(x_{\sigma\left(1\right)}^{-1};xx_1x_2\right)\nonumber\\
&\quad-\left(-1\right)^q\sum_{k=0}^{p_1+p_2-1}\binom{q+p_1+p_2-k-2}{q-1}\left(\left(-1\right)^k\Li_{k+1}\left(x\right)-\Li_{k+1}\left(x^{-1}\right)\right)\Li_{q+p_1+p_2-k-1}\left(xx_1x_2\right)\nonumber\\
&\quad-\left(-1\right)^q\sum_{k_1+k_2=q,\atop k_1,k_2\ge0}\binom{k_1+p_1-1}{p_1-1}
\binom{k_2+p_2-1}{p_2-1}\Li_{k_1+p_1}\left(x_1\right)\Li_{k_2+p_2}\left(x_2\right)\nonumber\\
&\quad-\left(-1\right)^{p_2+q}\binom{p_1+p_2+q-1}{p_1-1}\Li_{p_1+p_2+q}\left(x_1\right)-\left(-1\right)^{p_1+q}\binom{p_1+p_2+q-1}{p_2-1}\Li_{p_1+p_2+q}\left(x_2\right)\nonumber\\
&\quad-\sum_{k_1+k_2+k_3=q-1,\atop k_1,k_2,k_3\ge0}\left(\left(-1\right)^{k_1}\Li_{k_1+1}\left(x\right)-\Li_{k_1+1}\left(x^{-1}\right)\right)\left(\binom{k_2+p_1-1}{p_1-1}\left(-1\right)^{k_2}\Li_{k_2+p_1}\left(x_1\right)\right)\nonumber\\
&\qquad\times\left(\binom{k_3+p_2-1}{p_2-1}\left(-1\right)^{k_3}\Li_{k_3+p_2}\left(x_2\right)\right)+\left(-1\right)^{p_1+p_2+q}\Li_{p_1+p_2+q}\left(x\right)+\Li_{p_1+p_2+q}\left(x^{-1}\right)\nonumber\\
&\quad-\sum_{k_1+k_2=p_1+q-1,\atop k_1,k_2\ge0}\left(\left(-1\right)^{k_1}\Li_{k_1+1}\left(x\right)-\Li_{k_1+1}\left(x^{-1}\right)\right)\left(\binom{k_2+p_2-1}{p_2-1}\left(-1\right)^{k_2}\Li_{k_2+p_2}\left(x_2\right)\right)\nonumber\\
&\quad-\sum_{k_1+k_2=p_2+q-1,\atop k_1,k_2\ge0}\left(\left(-1\right)^{k_1}\Li_{k_1+1}\left(x\right)-\Li_{k_1+1}\left(x^{-1}\right)\right)\left(\binom{k_2+p_1-1}{p_1-1}\left(-1\right)^{k_2}\Li_{k_2+p_1}\left(x_1\right)\right).
\end{align}
\end{thm}
\begin{proof}
Letting
\[\oint\limits_{\left( \infty  \right)}F_{p_1p_2,q}(s)ds:= \oint\limits_{\left( \infty  \right)}\frac{\Phi(s;x)\phi^{(p_1-1)}(s;x_1)\phi^{(p_2-1)}(s;x_2)}{(p_1-1)!(p_2-1)!s^q} (-1)^{p_1+p_2}ds=0.\]
Clearly, the function $F_{p_1p_2,q}(s)$ only singularities are poles at the integers. At a positive integer $n\in \N$, the pole $s=n$ is simple and by the expansions \eqref{Lexp-phi-n-diffp-1} and \eqref{LEPhi-function}, the residue is
\begin{align*}
	\Res\left(F_{p_1,p_2,q}(s),n\right)=&\lim\limits_{s\to n}\left(s-n\right)x^{-n} \left(\frac1{s-n}+\sum_{m=0}^\infty \Big((-1)^m\Li_{m+1}(x)-\Li_{m+1}\Big(x^{-1}\Big)\Big)(s-n)^m \right)\\
	&\times x_1^{-n}\sum_{k=0}^\infty \binom{k+p_1-1}{p_1-1} (-1)^k  \left( \Li_{k+p_1}(x_1)-\zeta_{n-1}\Big(k+p_1;x_1\Big)\right)(s-n)^k\\
	&\times x_2^{-n}\sum_{k=0}^\infty \binom{k+p_2-1}{p_2-1} (-1)^k  \left( \Li_{k+p_2}(x_2)-\zeta_{n-1}\Big(k+p_2;x_2\Big)\right)(s-n)^ks^{-q}\\
	=&\lim\limits_{s\to n}\left(xx_1x_2\right)^{-n}\left( \Li_{p_1}(x_1)-\zeta_{n-1}\Big(p_1;x_1\Big)\right)\left( \Li_{p_2}(x_2)-\zeta_{n-1}\Big(p_2;x_2\Big)\right)s^{-q}\\
	=&\left(xx_1x_2\right)^{-n}\left( \Li_{p_1}(x_1)-\zeta_{n-1}\Big(p_1;x_1\Big)\right)\left( \Li_{p_2}(x_2)-\zeta_{n-1}\Big(p_2;x_2\Big)\right)n^{-q}.
\end{align*}
For positive integer $n$, the pole $s=-n$ has order $p_1+p_2+1$. By \eqref{Lexp-phi--n-diffp-1} and \eqref{LEPhi-function}, the residue is
\begin{align*}
	&\Res\left(F_{p_1p_2,q}(s),-n\right)\\
	&=\frac1{\left(p_1+p_2\right)!} \lim_{s\rightarrow -n} \frac{d^{p_1+p_2}}{ds^{p_1+p_2}}\left((s+n)^{p_1+p_2+1} F_{p_1p_2,q}(s)\right)\\
	&=\frac1{\left(p_1+p_2\right)!} \lim_{s\rightarrow -n}\frac{d^{p_1+p_2}}{ds^{p_1+p_2}}\left(s+n\right)^{p_1+p_2+1}\\
	&\qquad\times\left\{\left(x_1^n\sum_{k=0}^\infty \binom{k+p_1-1}{p_1-1} \left((-1)^k \Li_{k+p_1}(x_1)+(-1)^{p_1}\zeta_n\Big(k+p_1;x_1^{-1}\Big)\right)(s+n)^k +\frac{x_1^n}{(s+n)^{p_1}}\right)\right.\\
	&\qquad\times\left(x_2^n\sum_{k=0}^\infty \binom{k+p_2-1}{p_2-1} \left((-1)^k \Li_{k+p}(x_2)+(-1)^p\zeta_n\Big(k+p_2;x_2^{-1}\Big)\right)(s+n)^k +\frac{x_2^n}{(s+n)^{p_2}}\right)\\
	&\qquad\times\left.\left(x^{n} \left(\frac1{s+n}+\sum_{m=0}^\infty \Big((-1)^m\Li_{m+1}(x)-\Li_{m+1}\Big(x^{-1}\Big)\Big)(s+n)^m \right)\right)s^{-q}\right\}\\
	&=\left(-1\right)^q\binom{p_1+p_2+q-1}{p_1+p_2}\left(xx_1x_2\right)^n/n^{q+p_1+p_2}+\left(xx_1x_2\right)^n\sum_{k=0}^{p_2}\binom{k+p_1-1}{p_1-1}\binom{q+p_2-k-1}{q-1}\\
	&\qquad\times\left((-1)^k \Li_{k+p_1}(x_1)+(-1)^{p_1}\zeta_n\Big(k+p_1;x_1^{-1}\Big)\right)\left(-1\right)^q/n^{p_2+q-k}\\
	&\quad+\left(xx_1x_2\right)^n\sum_{k=0}^{p_1}\binom{k+p_2-1}{p_2-1}\binom{q+p_1-k-1}{q-1}\\
	&\qquad\times\left((-1)^k \Li_{k+p_2}(x_2)+(-1)^{p_2}\zeta_n\Big(k+p_2;x_2^{-1}\Big)\right)\left(-1\right)^q/n^{p_1+q-k}\\
	&\quad+\sum_{0\le k_1+k_2\le p_1-1}\left(-1\right)^q\binom{q+p_1-k_1-k_2-2}{q-1}\Big((-1)^{k_1}\Li_{k_1+1}(x)-\Li_{k_1+1}\Big(x^{-1}\Big)\Big)\left(xx_1x_2\right)^n\\
	&\qquad\times\binom{k_2+p_2-1}{p_2-1} \left((-1)^{k_2} \Li_{k_2+p_2}(x_2)+(-1)^{p_2}\zeta_n\Big(k_2+p_2;x_2^{-1}\Big)\right)/n^{p_1+q-k_1-k_2-1}\\
	&\quad+\sum_{0\le k_1+k_2\le p_2-1}\left(-1\right)^q\binom{q+p_2-k_1-k_2-2}{q-1}\Big((-1)^{k_1}\Li_{k_1+1}(x)-\Li_{k_1+1}\Big(x^{-1}\Big)\Big)\left(xx_1x_2\right)^n\\
	&\qquad\times\binom{k_2+p_1-1}{p_1-1} \left((-1)^{k_2} \Li_{k_2+p_1}(x_1)+(-1)^{p_1}\zeta_n\Big(k_2+p_1;x_1^{-1}\Big)\right)/n^{p_2+q-k_1-k_2-1}\\
	&\quad+\left(-1\right)^q\left(xx_1x_2\right)^n\left(\Li_{p_1}\left(x_1\right)+\left(-1\right)^{p_1}\zeta_n\left(p_1;x_1^{-1}\right)\right)\left(\Li_{p_2}\left(x_2\right)+\left(-1\right)^{p_2}\zeta_n\left(p_2;x_2^{-1}\right)\right)/n^q\\
	&\quad+\left(-1\right)^q\left(xx_1x_2\right)^n\sum_{k=0}^{p_1+p_2-1}\binom{q+p_1+p_2-k-2}{q-1}\left(\left(-1\right)^k\Li_{k+1}\left(x\right)-\Li_{k+1}\left(x^{-1}\right)\right)/n^{q+p_1+p_2-k-1}.
\end{align*}
The pole $s=0$ has order $p_1+p_2+q+1$, the residue is
\begin{align*}
	&\Res\left(F_{p_1p_2,q}(s),0\right)\\
	&=\sum_{k_1+k_2=q,\atop k_1,k_2\geq 0}\binom{k_1+p_1-1}{p_1-1}\binom{k_2+p_2-1}{p_2-1}\left(-1\right)^q\Li_{k_1+p_1}\left(x_1\right)\Li_{k_2+p_2}\left(x_2\right)\\
	&\quad+\binom{p_1+p_2+q-1}{p_1-1}\left(-1\right)^{p_2+q}\Li_{p_1+p_2+q}\left(x_1\right)+\binom{p_1+p_2+q-1}{p_2-1}\left(-1\right)^{p_1+q}\Li_{p_1+p_2+q}\left(x_2\right)\\
	&\quad+\sum_{k_1+k_2+k_3=q-1,\atop k_1,k_2,k_3\geq 0}\left(\left(-1\right)^{k_1}\Li_{k_1+1}\left(x\right)-\Li_{k_1+1}\left(x^{-1}\right)\right)\left(\left(-1\right)^{k_2}\binom{k_2+p_1-1}{p_1-1}\Li_{k_2+p_1}\left(x_1\right)\right)\\
	&\qquad\times\left(\left(-1\right)^{k_3}\binom{k_3+p_2-1}{p_2-1}\Li_{k_3+p_2}\left(x_2\right)\right)+\left(-1\right)^{p_1+p_2+q-1}\Li_{p_1+p_2+q}\left(x\right)-\Li_{p_1+p_2+q}\left(x^{-1}\right)\\
	&\quad+\sum_{k_1+k_2=p_1+q-1,\atop k_1,k_2\geq 0}\left(\left(-1\right)^{k_1}\Li_{k_1+1}\left(x\right)-\Li_{k_1+1}\left(x^{-1}\right)\right)\left(\left(-1\right)^{k_2}\binom{k_2+p_2-1}{p_2-1}\Li_{k_2+p_2}\left(x_2\right)\right)\\
	&\quad+\sum_{k_1+k_2=p_2+q-1,\atop k_1,k_2\geq 0}\left(\left(-1\right)^{k_1}\Li_{k_1+1}\left(x\right)-\Li_{k_1+1}\left(x^{-1}\right)\right)\left(\left(-1\right)^{k_2}\binom{k_2+p_1-1}{p_1-1}\Li_{k_2+p_1}\left(x_1\right)\right).
\end{align*}
By Lemma \ref{lem-redisue-thm}, we know that
\[\sum_{n=1}^\infty \left(\Res\left(F_{p_1p_2,q}(s),n\right)+\Res\left(F_{p_1p_2,q}(s),-n\right)\right)+\Res\left(F_{p_1p_2,q}(s),0\right)=0.\]
Finally, combining these three contributions yields the statement of Theorem \ref{thm-quadratic-C-ES-one}.
\end{proof}

\begin{re}
Theorem \ref{thm-quadratic-C-ES-one} corresponds to \cite[Thm. 4.2]{Flajolet-Salvy}.
\end{re}

Setting $p_1=p_2=1$ in Theorem \ref{thm-quadratic-C-ES-one} yields the following corollary.
\begin{cor} Let $x,x_1,x_2$ be roots of unity, and $q\geq 1$ with $(q,xx_1x_2)\neq (1,1)$. We have
\begin{align*}
	&S_{1,1;q}\left(x_1,x_2;\left(xx_1x_2\right)^{-1}\right)+\left(-1\right)^qS_{1,1;q}\left(x^{-1}_1,x^{-1}_2;xx_1x_2\right)\\
	&=-\Li_{q+2}\left(x^{-1}\right)-\Li_{1}\left(x_1\right)\Li_{1}\left(x_2\right)\Li_{q}\left(\left(xx_1x_2\right)^{-1}\right)-\left(-1\right)^q\binom{q+1}{2}\Li_{q+2}\left(xx_1x_2\right)\\
	&\quad+S_{1;q+1}\left(x_1;\left(xx_1\right)^{-1}\right)+S_{1;q+1}\left(x_2;\left(xx_2\right)^{-1}\right)+\Li_{1}\left(x_1\right)S_{1;q}\left(x_2;\left(xx_1x_2\right)^{-1}\right)\\
	&\quad-\Li_{1}\left(x_1\right)\Li_{q+1}\left(xx_1\right)+\Li_{1}\left(x_2\right)S_{1;q}\left(x_1;\left(xx_1x_2\right)^{-1}\right)-\Li_{1}\left(x_2\right)\Li_{q+1}\left(xx_2\right)\\
	&\quad-\left(-1\right)^qq\Li_{1}\left(x_1\right)\Li_{q+1}\left(xx_1x_2\right)+\left(-1\right)^qqS_{1;q+1}\left(x_1^{-1};xx_1x_2\right)\\
	&\quad-\left(-1\right)^qq\Li_{1}\left(x_2\right)\Li_{q+1}\left(xx_1x_2\right)+\left(-1\right)^qqS_{1;q+1}\left(x_2^{-1};xx_1x_2\right)+\left(-1\right)^q\Li_2\left(x_1\right)\Li_q\left(xx_1x_2\right)\\
	&\quad+\left(-1\right)^qS_{2;q}\left(x^{-1}_1;xx_1x_2\right)+\left(-1\right)^q\Li_2\left(x_2\right)\Li_q\left(xx_1x_2\right)+\left(-1\right)^qS_{2;q}\left(x^{-1}_2;xx_1x_2\right)\\
	&\quad-\left(\Li_1\left(x\right)-\Li_1\left(x^{-1}\right)\right)\left(\left(-1\right)^q\Li_{1}\left(x_2\right)\Li_q\left(xx_1x_2\right)-\left(-1\right)^qS_{1;q}\left(x^{-1}_2;xx_1x_2\right)\right)\\
	&\quad-\left(\Li_1\left(x\right)-\Li_1\left(x^{-1}\right)\right)\left(\left(-1\right)^q\Li_{1}\left(x_1\right)\Li_q\left(xx_1x_2\right)-\left(-1\right)^qS_{1;q}\left(x^{-1}_1;xx_1x_2\right)\right)\\
	&\quad-\left(-1\right)^q\Li_1\left(x_1\right)\Li_1\left(x_2\right)\Li_q\left(xx_1x_2\right)+\left(-1\right)^q\Li_{1}\left(x_1\right)S_{1;q}\left(x_2^{-1};xx_1x_2\right)\\
	&\quad+\left(-1\right)^q\Li_{1}\left(x_2\right)S_{1;q}\left(x_1^{-1};xx_1x_2\right)-\left(-1\right)^qq\left(\Li_{1}\left(x\right)-\Li_1\left(x^{-1}\right)\right)\Li_{q+1}\left(xx_1x_2\right)\\
	&\quad+\left(-1\right)^q\left(\Li_{2}\left(x\right)+\Li_2\left(x^{-1}\right)\right)\Li_{q}\left(xx_1x_2\right)+\left(-1\right)^q\Li_{q+2}\left(x_1\right)+\left(-1\right)^q\Li_{q+2}\left(x_2\right)\\
	&\quad-\left(-1\right)^q\sum_{k_1+k_2=q,\atop k_1,k_2\ge0}\Li_{k_1+1}\left(x_1\right)\Li_{k_2+1}\left(x_2\right)+\left(-1\right)^q\Li_{q+2}\left(x\right)+\Li_{q+2}\left(x^{-1}\right)\\
	&\quad-\sum_{k_1+k_2+k_3=q-1,\atop k_1,k_2,k_3\ge0}\left(\left(-1\right)^{k_1}\Li_{k_1+1}\left(x\right)-\Li_{k_1+1}\left(x^{-1}\right)\right)\left(-1\right)^{k_2+k_3}\Li_{k_2+1}\left(x_1\right)\Li_{k_3+1}\left(x_2\right)\\
	&\quad-\sum_{k_1+k_2=q,\atop k_1,k_2\ge0}\left(\left(-1\right)^{k_1}\Li_{k_1+1}\left(x\right)-\Li_{k_1+1}\left(x^{-1}\right)\right)\left(-1\right)^{k_2}\Li_{k_2+1}\left(x_2\right)\\
	&\quad-\sum_{k_1+k_2=q,\atop k_1,k_2\ge0}\left(\left(-1\right)^{k_1}\Li_{k_1+1}\left(x\right)-\Li_{k_1+1}\left(x^{-1}\right)\right)\left(-1\right)^{k_2}\Li_{k_2+1}\left(x_1\right).
\end{align*}
\end{cor}

Furthermore, we provide the following examples for illustration.

\begin{exa}Setting $\left(p_1,p_2,q\right)=\left(1,2,2\right)$ in Theorem \ref{thm-quadratic-C-ES-one}, we have
	\begin{align*}
		&S_{1,2;2}\left(x_1,x_2;\left(xx_1x_2\right)^{-1}\right)-S_{1,2;2}\left(x^{-1}_1,x^{-1}_2;xx_1x_2\right)\\
		&=-\Li_{1}\left(x_1\right)\Li_{2}\left(x_2\right)\Li_{2}\left(\left(xx_1x_2\right)^{-1}\right)-\Li_5\left(x^{-1}\right)-4\Li_5\left(xx_1x_2\right)+S_{1;4}\left(x_1;\left(xx_1\right)^{-1}\right)\\
		&\quad+S_{2;3}\left(x_2;\left(xx_2\right)^{-1}\right)+\Li_1\left(x_1\right)S_{2;2}\left(x_2;\left(xx_1x_2\right)^{-1}\right)-\Li_1\left(x_1\right)\Li_{4}\left(\left(xx_1\right)^{-1}\right)\\
		&\quad+\Li_2\left(x_2\right)S_{1;2}\left(x_1;\left(xx_1x_2\right)^{-1}\right)-\Li_2\left(x_2\right)\Li_{3}\left(\left(xx_2\right)^{-1}\right)-3\Li_1\left(x_1\right)\Li_4\left(xx_1x_2\right)\\
		&\quad+3S_{1;4}\left(x_1^{-1};xx_1x_2\right)+2\Li_2\left(x_1\right)\Li_3\left(xx_1x_2\right)+2S_{2;3}\left(x_1^{-1};xx_1x_2\right)-\Li_3\left(x_1\right)\Li_2\left(xx_1x_2\right)\\
		&\quad+S_{3;2}\left(x_1^{-1};xx_1x_2\right)-2\Li_2\left(x_2\right)\Li_3\left(xx_1x_2\right)-2S_{2;3}\left(x_2^{-1};xx_1x_2\right)+2\Li_3\left(x_2\right)\Li_2\left(xx_1x_2\right)\\
		&\quad-2S_{3;2}\left(x_2^{-1};xx_1x_2\right)-\left(\Li_1\left(x\right)-\Li_1\left(x^{-1}\right)\right)\Li_2\left(x_2\right)\Li_2\left(xx_1x_2\right)\\
		&\quad-\left(\Li_1\left(x\right)-\Li_1\left(x^{-1}\right)\right)S_{2;2}\left({x_2}^{-1};xx_1x_2\right)-2\left(\Li_1\left(x\right)-\Li_1\left(x^{-1}\right)\right)
\Li_{1}\left(x_1\right)\Li_3\left(xx_1x_2\right)\\
		&\quad+\left(\Li_2\left(x\right)+\Li_2\left(x^{-1}\right)\right)\Li_{1}\left(x_1\right)\Li_2\left(xx_1x_2\right)+\left(\Li_1\left(x\right)-\Li_1\left(x^{-1}\right)\right)\Li_{2}\left(x_1\right)\Li_2\left(xx_1x_2\right)\\
		&\quad+2\left(\Li_1\left(x\right)-\Li_1\left(x^{-1}\right)\right)S_{1;3}\left(x^{-1}_1;xx_1x_2\right)-\left(\Li_2\left(x\right)+\Li_2\left(x^{-1}\right)\right)S_{1;2}
\left(x^{-1}_1;xx_1x_2\right)\\
		&\quad+\left(\Li_1\left(x\right)-\Li_1\left(x^{-1}\right)\right)S_{2;2}\left(x^{-1}_1;xx_1x_2\right)-\Li_1\left(x_1\right)\Li_2\left(x_2\right)\Li_2\left(xx_1x_2\right)\\
		&\quad-\Li_1\left(x_1\right)S_{2;2}\left(x_2^{-1};xx_1x_2\right)+\Li_2\left(x_2\right)S_{1;2}\left(x_1^{-1};xx_1x_2\right)-3\left(\Li_1\left(x\right)-\Li_1\left(x^{-1}\right)\right)
\Li_4\left(xx_1x_2\right)\\
		&\quad+2\left(\Li_2\left(x\right)+\Li_2\left(x^{-1}\right)\right)\Li_3\left(xx_1x_2\right)-\left(\Li_3\left(x\right)-\Li_3\left(x^{-1}\right)\right)\Li_2\left(xx_1x_2\right)\\
		&\quad-3\Li_1\left(x_1\right)\Li_4\left(x_2\right)-\Li_3\left(x_1\right)\Li_2\left(x_2\right)-2\Li_2\left(x_1\right)\Li_3\left(x_2\right)-\Li_5\left(x_1\right)+4\Li_5\left(x_2\right)\\
		&\quad+\left(\Li_2\left(x\right)+\Li_2\left(x^{-1}\right)\right)\Li_1\left(x_1\right)\Li_2\left(x_2\right)+\left(\Li_1\left(x\right)-\Li_1\left(x^{-1}\right)\right)
\Li_2\left(x_1\right)\Li_2\left(x_2\right)\\
		&\quad+2\left(\Li_1\left(x\right)-\Li_1\left(x^{-1}\right)\right)\Li_1\left(x_1\right)\Li_3\left(x_2\right)-\Li_5\left(x\right)+\Li_5\left(x^{-1}\right)\\
		&\quad-2\left(\Li_2\left(x\right)+\Li_2\left(x^{-1}\right)\right)\Li_3\left(x_2\right)-\left(\Li_3\left(x\right)-\Li_3\left(x^{-1}\right)\right)\Li_2\left(x_2\right)\\
		&\quad-3\left(\Li_1\left(x\right)-\Li_1\left(x^{-1}\right)\right)\Li_4\left(x_2\right)+\left(\Li_4\left(x\right)+\Li_4\left(x^{-1}\right)\right)\Li_1\left(x_1\right)\\
		&\quad+\left(\Li_3\left(x\right)-\Li_3\left(x^{-1}\right)\right)\Li_2\left(x_1\right)+\left(\Li_2\left(x\right)+\Li_2\left(x^{-1}\right)\right)\Li_3\left(x_1\right)\\
		&\quad+\left(\Li_1\left(x\right)-\Li_1\left(x^{-1}\right)\right)\Li_4\left(x_1\right).
	\end{align*}
Setting $\left(p_1,p_2,q\right)=\left(1, 1, 2\right)$ in Theorem \ref{thm-quadratic-C-ES-one}, we have
\begin{align*}
	&S_{1,1;2}\left(x_1,x_2;\left(xx_1x_2\right)^{-1}\right)+S_{1,1;2}\left(x^{-1}_1,x^{-1}_2;xx_1x_2\right)\\
	&=-\Li_{4}\left(x^{-1}\right)-\Li_{1}\left(x_1\right)\Li_{1}\left(x_2\right)\Li_{2}\left(\left(xx_1x_2\right)^{-1}\right)-3\Li_{4}\left(xx_1x_2\right)\\
	&\quad+S_{1;3}\left(x_1;\left(xx_1\right)^{-1}\right)+S_{1;3}\left(x_2;\left(xx_2\right)^{-1}\right)+\Li_{1}\left(x_1\right)S_{1;2}\left(x_2;\left(xx_1x_2\right)^{-1}\right)\\
	&\quad-\Li_{1}\left(x_1\right)\Li_{3}\left(\left(xx_1\right)^{-1}\right)+\Li_{1}\left(x_2\right)S_{1;2}\left(x_1;\left(xx_1x_2\right)^{-1}\right)-\Li_{1}\left(x_2\right)\Li_{3}\left(\left(xx_2\right)^{-1}\right)\\
	&\quad-2\Li_{1}\left(x_1\right)\Li_{3}\left(xx_1x_2\right)+2S_{1;3}\left(x_1^{-1};xx_1x_2\right)\\
	&\quad-2\Li_{1}\left(x_2\right)\Li_{3}\left(xx_1x_2\right)+2S_{1;3}\left(x_2^{-1};xx_1x_2\right)+\Li_2\left(x_1\right)\Li_2\left(xx_1x_2\right)\\
	&\quad+S_{2;2}\left(x^{-1}_1;xx_1x_2\right)+\Li_2\left(x_2\right)\Li_2\left(xx_1x_2\right)+S_{2;2}\left(x^{-1}_2;xx_1x_2\right)\\
	&\quad-\left(\Li_1\left(x\right)-\Li_1\left(x^{-1}\right)\right)\left(\Li_{1}\left(x_2\right)\Li_2\left(xx_1x_2\right)-S_{1;2}\left(x^{-1}_2;xx_1x_2\right)\right)\\
	&\quad-\left(\Li_1\left(x\right)-\Li_1\left(x^{-1}\right)\right)\left(\Li_{1}\left(x_1\right)\Li_2\left(xx_1x_2\right)-S_{1;2}\left(x^{-1}_1;xx_1x_2\right)\right)\\
	&\quad-\Li_1\left(x_1\right)\Li_1\left(x_2\right)\Li_2\left(xx_1x_2\right)+\Li_{1}\left(x_1\right)S_{1;2}\left(x_2^{-1};xx_1x_2\right)\\
	&\quad+\Li_{1}\left(x_2\right)S_{1;2}\left(x_1^{-1};xx_1x_2\right)-2\left(\Li_{1}\left(x\right)-\Li_1\left(x^{-1}\right)\right)\Li_{3}\left(xx_1x_2\right)\\
	&\quad+\left(\Li_{2}\left(x\right)+\Li_2\left(x^{-1}\right)\right)\Li_{2}\left(xx_1x_2\right)+\Li_{4}\left(x_1\right)+\Li_{4}\left(x_2\right)\\
	&\quad-\Li_1\left(x_1\right)\Li_3\left(x_2\right)-\Li_3\left(x_1\right)\Li_1\left(x_2\right)-\Li_2\left(x_1\right)\Li_2\left(x_2\right)+\Li_{4}\left(x\right)+\Li_{4}\left(x^{-1}\right)\\
	&\quad+\left(\Li_{2}\left(x\right)+\Li_{2}\left(x^{-1}\right)\right)\Li_{1}\left(x_1\right)\Li_{1}\left(x_2\right)+\left(\Li_{1}\left(x\right)-\Li_{1}\left(x^{-1}\right)\right)\Li_{2}\left(x_1\right)\Li_{1}\left(x_2\right)\\
	&\quad+\left(\Li_{1}\left(x\right)-\Li_{1}\left(x^{-1}\right)\right)\Li_{1}\left(x_1\right)\Li_{2}\left(x_2\right)-\left(\Li_{1}\left(x\right)-\Li_{1}\left(x^{-1}\right)\right)\Li_{3}\left(x_2\right)\\
	&\quad-\left(\Li_{2}\left(x\right)+\Li_{2}\left(x^{-1}\right)\right)\Li_{2}\left(x_2\right)-\left(\Li_{3}\left(x\right)-\Li_{3}\left(x^{-1}\right)\right)\Li_{1}\left(x_2\right)\\
	&\quad-\left(\Li_{1}\left(x\right)-\Li_{1}\left(x^{-1}\right)\right)\Li_{3}\left(x_1\right)-\left(\Li_{2}\left(x\right)+\Li_{2}\left(x^{-1}\right)\right)\Li_{2}\left(x_1\right)-\left(\Li_{3}\left(x\right)-\Li_{3}\left(x^{-1}\right)\right)\Li_{1}\left(x_1\right).
	\end{align*}
\end{exa}

Next, we employ the method of contour integration to investigate some results on generalized quadratic Euler sums.

\begin{thm}\label{thm-quadratic-G-ES-one} For positive integer $q$ and $x_1,x_2,x_3$ are arbitrary complex numbers with $0<|x_1|, |x_2|,|x_3|\leq 1$ and $x_1,x_2,x_3 \neq 1$, we have
\begin{align}
&\sum_{1\leq i<j\leq 3} S_{1,1;q}\Big(x_i^{-1},x_j^{-1};x_1x_2x_3\Big)\nonumber\\
&=S_{1,1;q}\Big(x_1^{-1},x_2^{-1};x_1x_2x_3\Big)+S_{1,1;q}\Big(x_1^{-1},x_3^{-1};x_1x_2x_3\Big)+S_{1,1;q}\Big(x_2^{-1},x_3^{-1};x_1x_2x_3\Big)\nonumber\\
&=q\left(S_{1;q+1}\Big(x_1^{-1};x_1x_2x_3\Big)+S_{1;q+1}\Big(x_2^{-1};x_1x_2x_3\Big)+S_{1;q+1}\Big(x_3^{-1};x_1x_2x_3\Big)\right)\nonumber\\&\quad
+S_{2;q}\Big(x_1^{-1};x_1x_2x_3\Big)+S_{2;q}\Big(x_2^{-1};x_1x_2x_3\Big)+S_{2;q}\Big(x_3^{-1};x_1x_2x_3\Big)
\nonumber\\&\quad+\left(\Li_1(x_1)+\Li_1(x_2)\right)S_{1;q}\Big(x^{-1}_3;x_1x_2x_3\Big)+\left(\Li_1(x_1)+\Li_1(x_3)\right)S_{1;q}\Big(x^{-1}_2;x_1x_2x_3\Big)
\nonumber\\&\quad+\left(\Li_1(x_2)+\Li_1(x_3)\right)S_{1;q}\Big(x^{-1}_1;x_1x_2x_3\Big)+\left(\Li_{q+2}(x_1)+\Li_{q+2}(x_2)+\Li_{q+2}(x_3)\right)\nonumber\\&\quad -\frac{q(q+1)}{2} \Li_{q+2}(x_1x_2x_3)-q\left(\Li_1(x_1)+\Li_1(x_2)+\Li_1(x_3)\right)\Li_{q+1}(x_1x_2x_3)
\nonumber\\&\quad-\left(\Li_1(x_1)\Li_1(x_2)+\Li_1(x_1)\Li_1(x_3)+\Li_1(x_2)\Li_1(x_3)\right)\Li_q(x_1x_2x_3)\nonumber\\&\quad +\left(\Li_2(x_1)+\Li_2(x_2)+\Li_2(x_3)\right)\Li_{q}(x_1x_2x_3)\nonumber\\&\quad
-\sum_{k_1+k_2=q,\atop k_1,k_2\geq 0} \Li_{k_1+1}(x_1)\Li_{k_2+1}(x_2)-\sum_{k_1+k_3=q,\atop k_1,k_3\geq 0} \Li_{k_1+1}(x_1)\Li_{k_3+1}(x_3)\nonumber\\
&\quad-\sum_{k_2+k_3=q,\atop k_2,k_3\geq 0} \Li_{k_2+1}(x_2)\Li_{k_3+1}(x_3)+\sum_{k_1+k_2+k_3=q-1,\atop k_1,k_2,k_3\geq 0} \Li_{k_1+1}(x_1)\Li_{k_2+1}(x_2)\Li_{k_3+1}(x_3).
\end{align}
\end{thm}
\begin{proof}
The proof of this theorem is based on considering the residue calculation of the following contour integral:
\[\oint\limits_{\left( \infty  \right)}G_{1^3,q}(s)ds:= \oint\limits_{\left( \infty  \right)}\frac{\phi(s;x_1)\phi(s;x_2)\phi(s;x_3)}{s^q} ds=0.\]
Clearly, the function $G_{1^3,q}(s)$ only singularities are poles at the non-positive integers. At a positive integer $n\in \N$, the pole $s=-n$ is a pole of order $3$. Applying \eqref{Lexp-phi--n}, we have

\begin{align*}
G_{1^3,q}(s)&=\frac1{(s+n)^3}\frac{(x_1x_2x_3)^n}{s^q}+\frac{\sum_{j=1}^3\left(\Li_{1}(x_j)-\zeta_n\Big(1;x_j^{-1}\Big)\right)}{(s+n)^2}\frac{(x_1x_2x_3)^n}{s^q}\\
&\quad+\frac{\sum_{1\leq i< j\leq 3} \left(\Li_1(x_i)\Li_1(x_j)+\zeta_n\Big(1;x_i^{-1}\Big)\zeta_n\Big(1;x_j^{-1}\Big)\right)}{s+n}\frac{(x_1x_2x_3)^n}{s^q}\\
&\quad-\frac{\sum_{\sigma\in \{(123),(132),(231)\}} \left(\Li_{1}(x_{\sigma(1)})+\Li_1(x_{\sigma(2)})\right)\zeta_n\Big(1;x_{\sigma(3)}^{-1}\Big)}{s+n}\frac{(x_1x_2x_3)^n}{s^q}\\
&\quad-\frac{\sum_{j=1}^3\left(\Li_{2}(x_j)+\zeta_n\Big(2;x_j^{-1}\Big)\right)}{s+n}\frac{(x_1x_2x_3)^n}{s^q}+\cdots.
\end{align*}
By a simple residue computation, one obtains
\begin{align*}
\Res\left(G_{1^3,q}(s),-n\right)&=\frac1{2!} \lim_{s\rightarrow -n} \frac{d^2}{ds^2}\left((s+n)^3G_{1^3,q}(s) \right)\\
&=(-1)^q \frac{q(q+1)}{2} \frac{(x_1x_2x_3)^n}{n^{q+2}}+(-1)^q q \sum_{j=1}^3\left(\Li_{1}(x_j)-\zeta_n\Big(1;x_j^{-1}\Big)\right)\frac{(x_1x_2x_3)^n}{n^{q+1}}\\
&\quad-(-1)^q\sum_{\sigma\in \{(123),(132),(231)\}} \left(\Li_{1}(x_{\sigma(1)})+\Li_1(x_{\sigma(2)})\right)\zeta_n\Big(1;x_{\sigma(3)}^{-1}\Big) \frac{(x_1x_2x_3)^n}{n^{q}}\\
&\quad+(-1)^q\sum_{1\leq i< j\leq 3} \left(\Li_1(x_i)\Li_1(x_j)+\zeta_n\Big(1;x_i^{-1}\Big)\zeta_n\Big(1;x_j^{-1}\Big)\right)\frac{(x_1x_2x_3)^n}{n^{q}}\\
&\quad-(-1)^q\sum_{j=1}^3 \left(\Li_{2}(x_j)+\zeta_n\Big(2;x_j^{-1}\Big)\right)\frac{(x_1x_2x_3)^n}{n^{q}}.
\end{align*}
The pole $s=0$ has order $q+3$, the residue is
\begin{align*}
\Res\left(G_{1^3,q}(s),0\right)&=\frac1{(q+2)!} \lim_{s\rightarrow 0} \frac{d^{q+2}}{ds^{q+2}}\left(s^{q+3} G_{1^3,q}(s)\right)\\
&=(-1)^{q+1} \sum_{j=1}^3 \Li_{q+2}(x_j)+(-1)^q\sum_{1\leq i<j\leq3}\sum_{k_i+k_j=q,\atop k_i,k_j\geq 0} \Li_{k_{i}+1}(x_i) \Li_{k_{j}+1}(x_j)\\
&\quad+(-1)^{q-1}\sum_{k_1+k_2+k_3=q-1,\atop k_1,k_2,k_3\geq 0} \Li_{k_1+1}(x_1)\Li_{k_2+1}(x_2)\Li_{k_3+1}(x_3).
\end{align*}
Finally, applying Lemma \ref{lem-redisue-thm} and combining these two contributions yields the statement of Theorem \ref{thm-quadratic-G-ES-one}.
\end{proof}

Setting $x_1=x_2=x_3=x$ in Theorem \ref{thm-quadratic-G-ES-one} gives the following corollary.

\begin{cor} For positive integer $q$ and complex number $x$ with $0<|x|\leq 1$ and $x\neq 1$, we have
\begin{align}\label{equ-GES-Quar11}
S_{1,1;q}\Big(x^{-1},x^{-1};x^3\Big)&=qS_{1;q+1}\Big(x^{-1};x^3\Big)+S_{2;q}\Big(x^{-1};x^3\Big)+2\Li_1(x)S_{1;q}\Big(x^{-1};x^3\Big)+\Li_{q+2}(x)\nonumber\\
&\quad-\frac{q(q+1)}{6}\Li_{q+2}(x^3)-q\Li_1(x)\Li_{q+1}(x^3)+\Li_2(x)\Li_q(x^3)-\Li_1^2(x)\Li_q(x^3)\nonumber\\
&\quad-\sum_{i+j=q,\atop i,j\geq 0} \Li_{i+1}(x)\Li_{j+1}(x)+\frac1{3}\sum_{i+j+k=q-1,\atop i,j,k\geq 0} \Li_{i+1}(x)\Li_{j+1}(x)\Li_{k+1}(x).
\end{align}
\end{cor}
If letting $x\rightarrow 1$, then through elementary calculations, we can obtain \cite[Thm.4.1]{Flajolet-Salvy}.

Finally, according to definition of cyclotomic quadratic Euler $T$-sums and cyclotomic triple $t$-values, for $(p_1,x_1),(q,x)\neq (1,1)$, we have
\begin{align*}
S_{p_1,p_2;q}(x_1,x_2;x)&=\sum_{n=1}^\infty \frac{\ze_n(p_1;x_1)\ze_n(p_2;x_2)}{n^q}x^n\\
&=\sum_{n=1}^\infty \frac{\left(\ze_n(p_1;x_1)-\Li_{p_1}(x_1)\right)\ze_n(p_2;x_2)}{n^q}x^n+\Li_{p_1}(x_1) \sum_{n=1}^\infty \frac{\ze_n(p_2;x_2)}{n^q}x^n\\
&=\sum_{n=1}^\infty \frac{\left(\ze_n(p_1;x_1)-\Li_{p_1}(x_1)\right)\ze_{n-1}(p_2;x_2)}{n^q}x^n+\sum_{n=1}^\infty \frac{\ze_n(p_1;x_1)-\Li_{p_1}(x_1)}{n^{p_2+q}}(x_2x)^n\\
&\quad+\Li_{p_1}(x_1) \sum_{n=1}^\infty \frac{\left(\ze_{n-1}(p_2;x_2)+\frac{x_2^n}{n^{p_2}}\right)}{n)^q}x^n\\
&=-\Li_{p_2,q,p_1}(x_2,x,x_1)-\Li_{p_2+q,p_1}(x_2x,x_1)+\Li_{p_1}(x_1)\left(\Li_{p_2,q}(x_2,x)+\Li_{p_2+q}(x_2x)\right).
\end{align*}
Therefore, we can derive the following corollary regarding the parity of cyclotomic triple zeta values with a direct calculation.
\begin{cor}\label{cor-quadratic-C-MZV-one}
Let $x,y,z$ be roots of unity, and $p,m,q\geq 1$ with $(p,x), (q,y)$ and $ (m,z)\neq (1,1)$. Then
\begin{align*}
\Li_{p,q,m}(x,y,z)+(-1)^{p+q+m}\Li_{p,q,m}\Big(x^{-1},y^{-1},z^{-1}\Big)
\end{align*}
reduces to a combination of cyclotomic double zeta values and cyclotomic single zeta values.
\end{cor}

\section{Parity Results for Cyclotomic Cubic and Higher Order Euler Sums}

In this section, we employ the method of contour integration to derive explicit formulas for the parity of third-order cyclotomic Euler sums and provide a theorem concerning the parity of cyclotomic Euler sums of arbitrary order. We shall first provide the parity formulas for two cyclotomic or generalized cubic Euler sums.

\begin{thm}\label{thm-parityc-triple-ES-one} Let $x,x_1,x_2,x_3$ be roots of unity with $x_1,x_2,x_3\neq 1$, and $p\in \N$ with $(q,xx_1x_2x_3)\neq (1,1)$, we have
\begin{align}\label{equ-parityc-triple-ES-one}
	&S_{1,1,1;q}\left(x_1,x_2,x_3;\frac{1}{xx_1x_2x_3}\right)+\left(-1\right)^qS_{1,1,1;q}\left(x_{1}^{-1},x_{2}^{-1},x_{3}^{-1};{xx_1x_2x_3}\right)\nonumber\\
&=\Li_q\left(\frac{1}{xx_1x_2x_3}\right)\prod_{i=1}^{3}\left(\Li_1\left(x_i\right)\right)+\left(-1\right)^q\Li_q\left(xx_1x_2x_3\right)\prod_{i=1}^{3}\left(\Li_1\left(x_i\right)\right)\nonumber\\
&\quad-\sum_{\sigma\in\left\{\left(123\right),\left(132\right),\left(231\right)\right\}}\Li_1\left(x_{\sigma\left(1\right)}\right)\Li_1\left(x_{\sigma\left(2\right)}\right)
\left(S_{1;q}\left(x_{\sigma\left(3\right)};\frac{1}{xx_1x_2x_3}\right)-\Li_{q+1}\left(\frac{x_{\sigma\left(3\right)}}{xx_1x_2x_3}\right)\right)\nonumber\\
&\quad+\sum_{\sigma\in\left\{\left(123\right),\left(213\right),\left(312\right)\right\}}\Li_1\left(x_{\sigma\left(1\right)}\right)
\begin{Bmatrix}	S_{1,1;q}\left(x_{\sigma\left(2\right)},x_{\sigma\left(3\right)};\frac{1}{xx_1x_2x_3}\right)-S_{1;q+1}\left(x_{\sigma\left(2\right)};\frac{x_{\sigma\left(3\right)}}{xx_1x_2x_3}\right)\\
-S_{1;q+1}\left(x_{\sigma\left(3\right)};\frac{x_{\sigma\left(2\right)}}{xx_1x_2x_3}\right)+\Li_{q+2}\left(\frac{x_{\sigma\left(2\right)}x_{\sigma\left(3\right)}}{xx_1x_2x_3}\right)
\end{Bmatrix}\nonumber\\
&\quad+\sum_{\sigma\in\left\{\left(123\right),\left(213\right),\left(312\right)\right\}}S_{1,1;q+1}\left(x_{\sigma\left(1\right)},x_{\sigma\left(2\right)};\frac{x_{\sigma\left(3\right)}}{xx_1x_2x_3}\right)+\left(-1\right)^q\binom{q+2}{3}\Li_{q+3}\left(xx_1x_2x_3\right)\nonumber\\
&\quad-\sum_{\sigma\in\left\{\left(123\right),\left(213\right),\left(312\right)\right\}}S_{1;q+2}\left(x_{\sigma\left(1\right)};\frac{x_{\sigma\left(2\right)}x_{\sigma\left(3\right)}}{xx_1x_2x_3}\right)+\Li_{q+3}\left(\frac{1}{x}\right)\nonumber\\
&\quad+\sum_{i=1}^{3}\sum_{k=0}^{2}\binom{q-k+1}{q-1}\left(-1\right)^{k+q}\Li_{k+1}\left(x_{i}\right)\Li_{q-k+2}\left(xx_{1}x_2x_3\right)\nonumber\\
&\quad-\sum_{i=1}^{3}\sum_{k=0}^{2}\binom{q-k+1}{q-1}\left(-1\right)^{q}S_{k+1;q-k+2}\left(x_{i}^{-1};xx_{1}x_2x_3\right)\nonumber\\
&\quad+\sum_{\sigma\in\left\{\left(23\right),\left(12\right),\left(13\right)\right\}}\sum_{0\le k_1+k_2\le1}\left(-1\right)^{k_1+k_2+q}\binom{q-k_1-k_2}{q-1}\nonumber\\
&\qquad\qquad\qquad\qquad\times\Li_{k_1+1}\left(x_{\sigma\left(1\right)}\right)\Li_{k_2+1}\left(x_{\sigma\left(2\right)}\right)\Li_{q-k_1-k_2+1}\left(xx_1x_2x_3\right)\nonumber\\
&\quad+\sum_{\sigma\in\left\{\left(23\right),\left(12\right),\left(13\right)\right\}}\sum_{0\le k_1+k_2\le1}\binom{q-k_1-k_2}{q-1}\nonumber\\
&\qquad\qquad\qquad\times\begin{Bmatrix}
	\left(-1\right)^{k_1+q+1}\Li_{k_1+1}\left(x_{\sigma\left(1\right)}\right)S_{k_2+1;q-k_1-k_2+1}\left(x_{\sigma\left(2\right)}^{-1};xx_1x_2x_3\right)\\
	+\left(-1\right)^{k_2+q+1}\Li_{k_2+1}\left(x_{\sigma\left(2\right)}\right)S_{k_1+1;q-k_1-k_2+1}\left(x_{\sigma\left(1\right)}^{-1};xx_1x_2x_3\right)
\end{Bmatrix}\nonumber\\
&\quad+\sum_{\sigma\in\left\{\left(23\right),\left(12\right),\left(13\right)\right\}}\sum_{0\le k_1+k_2\le1}\left(-1\right)^{q}\binom{q-k_1-k_2}{q-1}S_{k_1+1,k_2+1;q-k_1-k_2+1}\left(x_{\sigma\left(1\right)}^{-1},x_{\sigma\left(2\right)}^{-1};xx_1x_2x_3\right)\nonumber\\
&\quad-\sum_{\sigma\in\left\{\left(123\right),\left(132\right),\left(231\right)\right\}}\left(-1\right)^q\Li_1\left({x_{\sigma\left(1\right)}}\right)
\Li_1\left({x_{\sigma\left(2\right)}}\right)S_{1;q}\left({x_{\sigma\left(3\right)}^{-1}};xx_1x_2x_3\right)\nonumber\\
&\quad+\sum_{\sigma\in\left\{\left(123\right),\left(213\right),\left(312\right)\right\}}\left(-1\right)^q\Li_1\left({x_{\sigma\left(1\right)}}\right)
S_{1,1;q}\left(x_{\sigma\left(2\right)}^{-1},x_{\sigma\left(3\right)}^{-1};xx_1x_2x_3\right)\nonumber\\
&\quad+\sum_{i=1}^{3}\sum_{0\le k_1+k_2\le1}\left(-1\right)^{k_2+q}\binom{q-k_1-k_2}{q-1}\left(\left(-1\right)^{k_1}\Li_{k_1+1}\left(x\right)-\Li_{k_1+1}\left(x^{-1}\right)\right)\nonumber\\
&\qquad\qquad\qquad\qquad\times\Li_{k_2+1}\left(x_i\right)\Li_{q-k_1-k_2+1}\left(xx_1x_2x_3\right)\nonumber\\
&\quad-\sum_{i=1}^{3}\sum_{0\le k_1+k_2\le1}\left(-1\right)^{q}\binom{q-k_1-k_2}{q-1}\left(\left(-1\right)^{k_1}\Li_{k_1+1}\left(x\right)-\Li_{k_1+1}\left(x^{-1}\right)\right)\nonumber\\
&\qquad\qquad\qquad\qquad\times S_{k_2+1;q-k_1-k_2+1}\left(x_i^{-1};xx_1x_2x_3\right)\nonumber\\
&\quad+\left(-1\right)^q\sum_{\sigma\in\left\{\left(23\right),\left(13\right),\left(12\right)\right\}}\left(\Li_{1}\left(x\right)-\Li_{1}\left(x^{-1}\right)\right)\Li_{1}\left(x_{\sigma\left(1\right)}\right)\Li_{1}\left(x_{\sigma\left(2\right)}\right)\Li_{q}\left(xx_{1}x_2x_3\right)\nonumber\\
&\quad-\left(-1\right)^q\sum_{\sigma\in\left\{\left(23\right),\left(13\right),\left(12\right),\atop\left(32\right),\left(31\right),\left(21\right)\right\}}
\left(\Li_{1}\left(x\right)-\Li_{1}\left(x^{-1}\right)\right)\Li_{1}\left(x_{\sigma\left(1\right)}\right)S_{1;q}\left(x^{-1}_{\sigma\left(2\right)};xx_{1}x_2x_3\right)\nonumber\\
&\quad+\left(-1\right)^q\sum_{\sigma\in\left\{\left(23\right),\left(13\right),\left(12\right)\right\}}\left(\Li_{1}\left(x\right)-\Li_{1}\left(x^{-1}\right)\right)
S_{1,1;q}\left(x^{-1}_{\sigma\left(1\right)},x^{-1}_{\sigma\left(2\right)};xx_{1}x_2x_3\right)\nonumber\\
&\quad+\left(-1\right)^q\sum_{k=0}^{2}\binom{q-k+1}{q-1}\left(\left(-1\right)^k\Li_{k+1}\left(x\right)-\Li_{k+1}\left(x^{-1}\right)\right)\Li_{q-k+2}\left(xx_{1}x_2x_3\right)\nonumber\\
&\quad+\left(-1\right)^q\Li_{q+3}\left(x\right)-\Li_{q+3}\left(x^{-1}\right)+\sum_{i=1}^{3}\left(-1\right)^q\Li_{q+3}\left(x_i\right)\nonumber\\
&\quad+\sum_{k_1+k_2+k_3=q,\atop k_1,k_2,k_3\ge0}\left(-1\right)^q\Li_{k_1+1}\left(x_1\right)\Li_{k_2+1}\left(x_2\right)\Li_{k_3+1}\left(x_3\right)\nonumber\\
&\quad+\sum_{\sigma\in\left\{\left(12\right),\left(13\right),\left(23\right)\right\}}\sum_{k_1+k_2=q+1,\atop k_1,k_2\ge0}\left(-1\right)^{q+1}\Li_{k_1+1}\left(x_{\sigma\left(1\right)}\right)\Li_{k_2+1}\left(x_{\sigma\left(2\right)}\right)\nonumber\\
&\quad+\sum_{i=1}^{3}\sum_{k_1+k_2=q+1,\atop k_1,k_2\ge0}\left(-1\right)^{k_1}\Li_{k_1+1}\left(x_{i}\right)\left(\left(-1\right)^{k_2}\Li_{k_2+1}\left(x\right)-\Li_{k_2+1}\left(x^{-1}\right)\right)\nonumber\\
&\quad+\sum_{\sigma\in\left\{\left(12\right),\left(13\right),\left(23\right)\right\}}\sum_{k_1+k_2+k_3=q,\atop k_1,k_2,k_3\ge0}\left(-1\right)^{k_1+k_2}\Li_{k_1+1}\left(x_{\sigma\left(1\right)}\right)\Li_{k_2+1}\left(x_{\sigma\left(2\right)}\right)\nonumber\\
&\qquad\qquad\qquad\qquad\times\left(\left(-1\right)^{k_3}\Li_{k_3+1}\left(x\right)-\Li_{k_3+1}\left(x^{-1}\right)\right)\nonumber\\
&\quad+\sum_{k_1+k_2+k_3+k_4=q-1,\atop k_1,k_2,k_3,k_4\ge0}\left(-1\right)^{k_1+k_2}\Li_{k_1+1}\left(x_{1}\right)\Li_{k_2+1}\left(x_{2}\right)\Li_{k_3+1}\left(x_{3}\right)\nonumber\\
&\qquad\qquad\qquad\qquad\times\left(\left(-1\right)^{k_4}\Li_{k_4+1}\left(x\right)-\Li_{k_4+1}\left(x^{-1}\right)\right).
\end{align}
\end{thm}
\begin{proof}
Consider the contour integration
\begin{align*}
\oint\limits_{\left( \infty  \right)}F_{1^3,q}(s)ds:=\oint\limits_{\left( \infty  \right)} \frac{\Phi(s;x)\phi(s;x_1)\phi(s;x_2)\phi(s;x_3)}{s^q}ds=0.
\end{align*}
Clearly, the function $F_{1^3,q}(s)$ only singularities are poles at the integers. At a positive integer $n\in \N$, the pole $s=n$ is simple and by the expansions \eqref{Lexp-phi-n} and \eqref{LEPhi-function}, the residue is
\begin{align*}
	\Res\left(F_{1^3,q}(s),n\right)=&\lim\limits_{s\to n}\left(s-n\right)x^{-n} \left(\frac1{s-n}+\sum_{m=0}^\infty \Big((-1)^m\Li_{m+1}(x)-\Li_{m+1}\Big(x^{-1}\Big)\Big)(s-n)^m \right)\\
	&\times x_1^{-n}\sum_{k=0}^\infty  (-1)^k  \left( \Li_{k+1}(x_1)-\zeta_{n-1}\Big(k+1;x_1\Big)\right)(s-n)^k\\
	&\times x_2^{-n}\sum_{k=0}^\infty  (-1)^k  \left( \Li_{k+1}(x_2)-\zeta_{n-1}\Big(k+1;x_2\Big)\right)(s-n)^k\\
	&\times x_3^{-n}\sum_{k=0}^\infty  (-1)^k  \left( \Li_{k+1}(x_3)-\zeta_{n-1}\Big(k+1;x_3\Big)\right)(s-n)^ks^{-q}\\
	=&\left(xx_1x_2x_3\right)^{-n}\prod_{i=1}^{3}\left( \Li_{1}(x_i)-\zeta_{n-1}\Big(1;x_i\Big)\right)n^{-q}.
\end{align*}
For positive integer $n$, the pole $s=-n$ has order $4$. By \eqref{Lexp-phi--n} and \eqref{LEPhi-function}, after some rather lengthy calculations, we obtain the residue as
\begin{align*}
	&\Res\left(F_{1^3,q}(s),-n\right)=\frac1{3!} \lim_{s\rightarrow -n} \frac{d^{3}}{ds^{3}}\left((s+n)^{4} F_{1^3,q}(s)\right)\\
	&=\frac1{3!} \lim_{s\rightarrow -n} \frac{d^{3}}{ds^{3}}(s+n)^{4}\\
	&\quad\times\left\{\left(x_1^n\sum_{k=0}^\infty \left((-1)^k \Li_{k+1}(x_1)-\zeta_n\Big(k+1;x_1^{-1}\Big)\right)(s+n)^k +\frac{x_1^n}{s+n}\right)\right.\\
	&\quad\times\left(x_2^n\sum_{k=0}^\infty \left((-1)^k \Li_{k+1}(x_2)-\zeta_n\Big(k+1;x_2^{-1}\Big)\right)(s+n)^k +\frac{x_2^n}{s+n}\right)\\
	&\quad\times\left.\left(x_3^n\sum_{k=0}^\infty \left((-1)^k \Li_{k+1}(x_1)-\zeta_n\Big(k+1;x_3^{-1}\Big)\right)(s+n)^k +\frac{x_3^n}{s+n}\right)s^{-q}\right\}\\
	&=\left(-1\right)^q\binom{q+2}{3}\left(xx_1x_2x_3\right)^n/n^{q+3}\\
	&\quad+\left(xx_1x_2x_3\right)^n\sum_{i=1}^{3}\sum_{k=0}^{2}\left(-1\right)^q\binom{q+1-k}{q-1}\left(\left(-1\right)^k\Li_{k+1}\left(x_i\right)-\zeta_n\left(k+1;x_i^{-1}\right)\right)/n^{q-k+2}\\
	&\quad+\sum_{\sigma\in\left\{\left(23\right),\left(12\right),\left(13\right)\right\}}\sum_{0\le k_1+k_2\le1}\left(xx_1x_2x_3\right)^n\left(\left(-1\right)^{k_1}\Li_{k_1+1}\left(x_{\sigma\left(1\right)}\right)-\zeta_n\left(k_1+1;x_{\sigma\left(1\right)}^{-1}\right)\right)\\
	&\qquad\times\left(\left(-1\right)^{k_2}\Li_{k_2+1}\left(x_{\sigma\left(2\right)}\right)-\zeta_n\left(k_2+1;x_{\sigma\left(2\right)}^{-1}\right)\right)\binom{q-k_1-k_2}{q-1}\left(-1\right)^q/n^{q-k_1-k_2+1}\\
	&\quad+\left(xx_1x_2x_3\right)^n\prod_{i=1}^{3}\left(\Li_{1}\left(x_i\right)-\zeta_n\left(1;x_i^{-1}\right)\right)\left(-1\right)^q/n^q\\
	&\quad+\left(xx_1x_2x_3\right)^n\sum_{i=1}^{3}\sum_{0\le k_1+k_2\le1}\left(-1\right)^q\binom{q-k_1-k_2}{q-1}\left(\left(-1\right)^{k_1}\Li_{k_1+1}\left(x\right)-\Li_{k_1+1}\left(x^{-1}\right)\right)\\
	&\qquad\times\left(\left(-1\right)^{k_2}\Li_{k_2+1}\left(x_i\right)-\zeta_n\left(k_2+1;x_i^{-1}\right)\right)/n^{q-k_1-k_2+1}\\
	&\quad+\sum_{\sigma\in\left\{\left(23\right),\left(13\right),\left(12\right)\right\}}\left(xx_1x_2x_3\right)^n\left(\Li_1\left(x\right)-\Li_{1}\left(x^{-1}\right)\right)\left(\Li_1\left(x_{\sigma\left(1\right)}\right)-\zeta_{n}\left(1;x_{\sigma\left(1\right)}^{-1}\right)\right)\\
	&\qquad\times\left(\Li_1\left(x_{\sigma\left(2\right)}\right)-\zeta_{n}\left(1;x_{\sigma\left(2\right)}^{-1}\right)\right)\left(-1\right)^q/n^q\\
	&\quad+\sum_{k=0}^{2}\left(xx_1x_2x_3\right)^n\binom{q-k+1}{q-1}\left(\left(-1\right)^k\Li_{k+1}-\Li_{k+1}\left(x^{-1}\right)\right)\left(-1\right)^q/n^{q-k+2}.
\end{align*}
The pole $s=0$ has order $q+4$, the residue is
\begin{align*}
	&\Res\left(F_{1^3,q}(s),0\right)\\
	&=\left(-1\right)^q\Li_{q+3}\left(x\right)-\Li_{q+3}\left(x^{-1}\right)+\sum_{i=1}^{3}\left(-1\right)^q\Li_{q+3}\left(x_i\right)\\
	&\quad+\sum_{k_1+k_2+k_3=q,\atop k_1,k_2,k_3\ge0}\left(-1\right)^q\Li_{k_1+1}\left(x_1\right)\Li_{k_2+1}\left(x_2\right)\Li_{k_3+1}\left(x_3\right)\\
	&\quad+\sum_{\sigma\in\left\{\left(12\right),\left(13\right),\left(23\right)\right\}}\left(-1\right)^{q+1}\Li_{k_1+1}\left(x_{\sigma\left(1\right)}\right)\Li_{k_2+1}\left(x_{\sigma\left(2\right)}\right)\\
	&\quad+\sum_{i=1}^{3}\sum_{k_1+k_2=q+1,\atop k_1,k_2\geq 0}\left(-1\right)^{k_1}\Li_{k_1+1}\left(x_i\right)\left(\left(-1\right)^{k_2}\Li_{k_2+1}\left(x\right)-\Li_{k_2+1}\left(x^{-1}\right)\right)\\
	&\quad+\sum_{\sigma\in\left\{\left(23\right),\left(13\right),\left(12\right)\right\}}\sum_{k_1+k_2+k_3=q,\atop k_1,k_2,k_3\geq 0}\left(-1\right)^{k_1+k_2}\Li_{k_1+1}\left(x_{\sigma\left(1\right)}\right)\Li_{k_2+1}\left(x_{\sigma\left(2\right)}\right)\\
	&\qquad\times\left(\left(-1\right)^{k_3}\Li_{k_3+1}\left(x\right)-\Li_{k_3+1}\left(x^{-1}\right)\right)\\
	&\quad+\sum_{k_1+k_2+k_3+k_4=q-1,\atop k_1,k_2,k_3,k_4\geq 0}\left(-1\right)^{k_1+k_2+k_3}\Li_{k_1+1}\left(x_{1}\right)\Li_{k_2+1}\left(x_{2}\right)\Li_{k_3+1}\left(x_{3}\right)\\
	&\qquad\times\left(\left(-1\right)^{k_4}\Li_{k_4+1}\left(x\right)-\Li_{k_4+1}\left(x^{-1}\right)\right).
\end{align*}
By Lemma \ref{lem-redisue-thm}, we know that
\[\sum_{n=1}^\infty \left(\Res\left(F_{1^3,q}(s),n\right)+\Res\left(F_{1^3,q}(s),-n\right)\right)+\Res\left(F_{1^3,q}(s),0\right)=0.\]
Finally, combining these three contributions yields the statement of Theorem \ref{thm-parityc-triple-ES-one}.
\end{proof}

\begin{exa}As an example, considering $x=q=1$ and $x_1=x_2=x_3=-1$ in Theorem \ref{thm-parityc-triple-ES-one}, we have
	\begin{align*}
		&3S_{1,2;1}(-1,-1;-1)+3S_{2,1;1}(-1,-1;-1)-3S_{1,1;2}(-1,-1;1)+3S_{1,1;2}(-1,-1;-1)\\
		&=3\log^2(2)\zeta(2)+6\log(2)S_{1;2}(-1;1)-4\zeta(\bar{4})-\zeta(4)+6\zeta^2(\bar{2})-3\log(2)\zeta(\bar{3})+3S_{2;2}(-1;-1)\\
		&\quad+3S_{3;1}(-1;-1)-6\log(2)S_{1;2}(-1;-1)-6\zeta(\bar{2})S_{1;1}(-1;-1)-6\log(2)S_{2;1}(-1;-1)\\
		&\quad-6\zeta(2)S_{1;1}(-1;-1)+8\zeta(\bar{2})\zeta(2).
	\end{align*}
Applying \eqref{TripleCESTPL} gives
	\begin{align*}
	&6\zeta(\bar{2},\bar{1},\bar{1})+6\zeta(\bar{1},\bar{2},\bar{1})+6\zeta(\bar{1},\bar{1},\bar{2})-6\zeta(\bar{1},\bar{1},2)\\
	&=6\zeta(\bar{1},\bar{3})+8\zeta(4)+6\zeta^2(\bar{2})+3\zeta(\bar{2},\bar{2})-6\log(2)\zeta(\bar{1},\bar{2})-12\log(2)\zeta(3)-6\log(2)\zeta(\bar{2},\bar{1})\\
	&\quad+6\zeta(\bar{2})\zeta(\bar{1},\bar{1})-6\zeta(\bar{2})\zeta(2)-3\zeta(2,\bar{2})-6\zeta(3,\bar{1})-6\zeta(\bar{2},2)-12\zeta(\bar{1},3)+3\log^2(2)\zeta(2)\\
	&\quad+6\log(2)\zeta(\bar{1},2)+9\log(2)\zeta(\bar{3})+3\zeta(\bar{3},\bar{1})+3\zeta(2,2)-6\zeta(2)\zeta(\bar{1},\bar{1})-6\zeta^2(2)+8\zeta(2)\zeta(\bar{2})\\
	&\quad-13\zeta(\bar{4})-6\log(2)\zeta(\bar{3}).
	\end{align*}
\end{exa}

\begin{thm}\label{thm-parity-G-ES-one} For positive integer $q$ and $x_1,x_2,x_3,x_4$ are arbitrary complex numbers with $0<|x_1|, |x_2|,|x_3|,|x_4|\leq 1$ and $x_1,x_2,x_3,x_4 \neq 1$, we have
\begin{align}
	&\left(-1\right)^q\binom{q+2}{3}\Li_{q+3}\left(x_1x_2x_3x_4\right)-\left(-1\right)^q\sum_{\sigma\in\left\{\left(234\right),\left(134\right),\atop\left(124\right),\left(123\right)\right\}}
S_{1,1,1;q}\left(x^{-1}_{\sigma\left(1\right)},x^{-1}_{\sigma\left(2\right)},x^{-1}_{\sigma\left(3\right)};x_1x_2x_3x_4\right)\nonumber\\
	&+\left(-1\right)^q\sum_{i=1}^{4}\sum_{k=0}^{2}\binom{q-k+1}{q-1}\left(\left(-1\right)^k\Li_{k+1}\left(x_i\right)\Li_{q-k+2}\left(x_1x_2x_3x_4\right)-
S_{k+1;q-k+2}\left(x_i^{-1};x_1x_2x_3x_4\right)\right)\nonumber\\
	&+\left(-1\right)^q\sum_{\sigma\in\left\{\left(34\right),\left(24\right),\left(23\right),\atop\left(14\right),\left(12\right),\left(13\right)\right\}}\sum_{0\le k_1+k_2\le1}\binom{q-k_1-k_2}{q-1}\nonumber\\
	&\quad\times\left(-1\right)^{k_1+k_2}\Li_{k_1+1}\left(x_{\sigma\left(1\right)}\right)\Li_{k_2+1}\left(x_{\sigma\left(2\right)}\right)\Li_{q-k_1-k_2+1}\left(x_1x_2x_3x_4\right)\nonumber\\
	&+\left(-1\right)^q\sum_{\sigma\in\left\{\left(34\right),\left(24\right),\left(23\right),\atop\left(14\right),\left(12\right),\left(13\right)\right\}}\sum_{0\le k_1+k_2\le1}\binom{q-k_1-k_2}{q-1}\nonumber\\
	&\quad\times\begin{Bmatrix}
		\left(-1\right)^{k_1+1}\Li_{k_1+1}\left(x_{\sigma\left(1\right)}\right)S_{k_2+1;q-k_1-k_2+1}\left(x_{\sigma\left(2\right)}^{-1};x_1x_2x_3x_4\right)\\
		+\left(-1\right)^{k_2+1}\Li_{k_2+1}\left(x_{\sigma\left(2\right)}\right)S_{k_1+1;q-k_1-k_2+1}\left(x_{\sigma\left(1\right)}^{-1};x_1x_2x_3x_4\right)
	\end{Bmatrix}\nonumber\\
	&+\left(-1\right)^q\sum_{\sigma\in\left\{\left(34\right),\left(24\right),\left(23\right),\atop\left(14\right),\left(12\right),\left(13\right)\right\}}\sum_{0\le k_1+k_2\le1}\binom{q-k_1-k_2}{q-1}S_{k_1+1,k_2+1;q-k_1-k_2+1}\left(x_{\sigma\left(1\right)}^{-1},x_{\sigma\left(2\right)}^{-1};x_1x_2x_3x_4\right)\nonumber\\
	&+\left(-1\right)^q\sum_{\sigma\in\left\{\left(234\right),\left(134\right),\atop\left(124\right),\left(123\right)\right\}}\Li_{1}\left(x_{\sigma\left(1\right)}\right)
\Li_{1}\left(x_{\sigma\left(2\right)}\right)\Li_{1}\left(x_{\sigma\left(3\right)}\right)\Li_{q}\left(x_1x_2x_3x_4\right)\nonumber\\
	&+\left(-1\right)^{q+1}\sum_{\sigma\in\left\{\left(234\right),\left(134\right),\left(124\right),\left(123\right),\left(243\right),\left(143\right),\atop\left(142\right),
\left(132\right),\left(342\right),\left(341\right),\left(241\right),\left(231\right)\right\}}\Li_{1}\left(x_{\sigma\left(1\right)}\right)\Li_{1}\left(x_{\sigma\left(2\right)}\right)
S_{1;q}\left(x_{\sigma\left(3\right)}^{-1};x_1x_2x_3x_4\right)\nonumber\\
	&+\left(-1\right)^q\sum_{\sigma\in\left\{\left(234\right),\left(134\right),\left(124\right),\left(123\right),\left(324\right),\left(314\right),\atop\left(214\right),
\left(213\right),\left(423\right),\left(413\right),\left(412\right),\left(312\right)\right\}}\Li_1\left(x_{\sigma\left(1\right)}\right)
S_{1,1;q}\left(x_{\sigma\left(2\right)}^{-1},x_{\sigma\left(3\right)}^{-1};x_1x_2x_3x_4\right)\nonumber\\
	&+\sum_{i=1}^{4}\left(-1\right)^q\Li_{q+3}\left(x_i\right)+\sum_{\sigma\in\left\{\left(34\right),\left(24\right),\left(23\right),\atop\left(14\right),\left(12\right),\left(13\right)\right\}}\sum_{k_1+k_2=q+1,\atop k_1,k_2\ge0}\left(-1\right)^{q+1}\Li_{k_1+1}\left(x_{\sigma\left(1\right)}\right)\Li_{k_2+1}\left(x_{\sigma\left(2\right)}\right)\nonumber\\
	&+\sum_{\sigma\in\left\{\left(234\right),\left(134\right),\atop\left(124\right),\left(123\right)\right\}}\sum_{k_1+k_2+k_3=q,\atop k_1,k_2,k_3\ge0}\left(-1\right)^{q}\Li_{k_1+1}\left(x_{\sigma\left(1\right)}\right)\Li_{k_2+1}\left(x_{\sigma\left(2\right)}\right)\Li_{k_3+1}\left(x_{\sigma\left(3\right)}\right)\nonumber\\
	&+\sum_{k_1+k_2+k_3+k_4=q-1,\atop k_1,k_2,k_3,k_4\ge0}\left(-1\right)^{q-1}\Li_{k_1+1}\left(x_{1}\right)\Li_{k_2+1}\left(x_{2}\right)\Li_{k_3+1}\left(x_{3}\right)\Li_{k_4+1}\left(x_{4}\right)=0
\end{align}
\end{thm}
\begin{proof}
Consider the contour integral
\[\oint\limits_{\left( \infty  \right)}G_{1^4,q}(s)ds:=\oint\limits_{\left( \infty  \right)}\frac{\phi(s;x_1)\phi(s;x_2)\phi(s;x_3)\phi(s;x_4)}{s^q}ds=0.\]
Clearly, the function $G_{1^4,q}(s)$ only singularities are poles at the non-positive integers. At a positive integer $n\in \N$, the pole $s=-n$ is a pole of order $4$. Applying \eqref{Lexp-phi--n} and residue computation, we have
\begin{align*}
	&\Res\left(G_{1^4,q}(s),-n\right)=\frac1{3!} \lim_{s\rightarrow -n} \frac{d^3}{ds^3}\left((s+n)^4G_{1^4,q}(s) \right)\\
	&=(-1)^q \binom{q+2}{3} \frac{(x_1x_2x_3x_4)^n}{n^{q+3}}\\
	&\quad+(-1)^q\sum_{i=1}^{4}\sum_{k=0}^2\binom{q-k+1}{q-1} \left(\left(-1\right)^k\Li_{k+1}(x_i)-\zeta_n\Big(k+1;x_i^{-1}\Big)\right)\frac{(x_1x_2x_3x_4)^n}{n^{q-k+2}}\\
	&\quad+(-1)^q\sum_{\sigma\in \left\{(34),(24),(23),\atop\left(14\right),\left(12\right),\left(13\right)\right\}}\sum_{0\le k_1+k_2\le1}\binom{q-k_1-k_2}{q-1}\left(\left(-1\right)^{k_1}\Li_{k_1+1}(x_{\sigma\left(1\right)})-\zeta_n\Big(k_1+1;x_{\sigma\left(1\right)}^{-1}\Big)\right) \\
	&\qquad\times\left(\left(-1\right)^{k_2}\Li_{k_2+1}(x_{\sigma\left(2\right)})-\zeta_n\Big(k_2+1;x_{\sigma\left(2\right)}^{-1}\Big)\right) \frac{(x_1x_2x_3x_4)^n}{n^{q-k_1-k_2+1}}\\
	&\quad+(-1)^q\sum_{\sigma\in \left\{(234),(134),\atop\left(124\right),\left(123\right)\right\}}\left(\Li_{1}(x_{\sigma\left(1\right)})-\zeta_n\Big(1;x_{\sigma\left(1\right)}^{-1}\Big)\right)\left(\Li_{1}(x_{\sigma\left(2\right)})-\zeta_n\Big(1;x_{\sigma\left(2\right)}^{-1}\Big)\right) \\
	&\qquad\times\left(\Li_{1}(x_{\sigma\left(3\right)})-\zeta_n\Big(1;x_{\sigma\left(3\right)}^{-1}\Big)\right)\frac{(x_1x_2x_3x_4)^n}{n^{q}}.
\end{align*}
The pole $s=0$ has order $q+4$, the residue is
\begin{align*}
	&\Res\left(G_{1^4,q}(s),0\right)=\frac1{(q+3)!} \lim_{s\rightarrow 0} \frac{d^{q+3}}{ds^{q+3}}\left(s^{q+4} G_{1^4,q}(s)\right)\\
	&=(-1)^{q} \sum_{i=1}^4 \Li_{q+3}(x_i)+(-1)^{q+1}\sum_{\sigma\in\left\{\left(34\right),\left(24\right),\left(23\right),\atop\left(14\right),\left(12\right),\left(13\right)\right\}}\sum_{k_1+k_2=q+1,\atop k_i,k_j\geq 0} \Li_{k_{1}+1}\left(x_{\sigma\left(1\right)}\right) \Li_{k_{2}+1}\left(x_{\sigma\left(2\right)}\right)\\
	&\quad+(-1)^{q}\sum_{\sigma\in\left\{\left(234\right),\left(134\right),\atop\left(124\right),\left(123\right)\right\}}\sum_{k_1+k_2+k_3=q,\atop k_1,k_2,k_3\geq 0} \Li_{k_1+1}\left(x_{\sigma\left(1\right)}\right)\Li_{k_2+1}\left(x_{\sigma\left(2\right)}\right)\Li_{k_3+1}\left(x_{\sigma\left(3\right)}\right)\\
	&\quad+\left(-1\right)^{q-1}\sum_{k_1+k_2+k_3+k_4=q-1,\atop k_1,k_2,k_3,k_4\geq 0}\Li_{k_1+1}\left(x_{1}\right)\Li_{k_2+1}\left(x_{2}\right)\Li_{k_3+1}\left(x_{3}\right)\Li_{k_4+1}\left(x_{4}\right).
\end{align*}
Finally, applying Lemma \ref{lem-redisue-thm} and combining these two contributions yields the statement of Theorem \ref{thm-parity-G-ES-one}.
\end{proof}

\begin{exa}As an example, considering $q=2$ and $x_1=x_2=x_3=x_4=-1$ in Theorem \ref{thm-parity-G-ES-one}, we have
	\begin{align*}
		&2S_{1,1,1;2}\left(-1,-1,-1;1\right)\\
		&=6S_{1,1;3}\left(-1,-1;1\right)+3S_{2,1;2}\left(-1,-1;1\right)+3S_{1,2;2}\left(-1,-1;1\right)+2\zeta(5)-6\log(2)\zeta(4)\\
		&\quad-6S_{1;4}(-1;1)-4\zeta(3)\zeta(\bar{2})-4S_{2;3}(-1;1)+2\zeta(\bar{3})\zeta(2)-2S_{3;2}(-1;1)+6\log^2(2)\zeta(3)\\
		&\quad+6\zeta(\bar{2})\zeta(2)\log(2)+12\log(2)S_{1;3}(-1;1)+6\zeta(\bar{2})S_{1;2}(-1;1)+6\log(2)S_{2;2}(-1;1)\\
		&\quad-2\log^3(2)\zeta(2)-6\log^2(-1)S_{1;2}(-1;1)-6\log(2)S_{1,1;2}(-1,-1;1)+2\zeta(\bar{5})+6\zeta(\bar{4})\log(2)\\
		&\quad-6\zeta(\bar{2})\zeta(\bar{3})+6\zeta(\bar{3})\log^2(2)-6\zeta^2(\bar{2})\log(2)+2\zeta(\bar{2})\log^3(2).
	\end{align*}
Applying the stuffle relations, the alternating Euler sums in the above expression can all be written in terms of alternating multiple zeta values, thus leading to the following result
\begin{align*}
	&6\zeta(\bar{1},\bar{1},\bar{1},2)\\
	&=-6\log(2)\zeta(4)-6\zeta(\bar{1},4)-2\zeta(\bar{2})\zeta(3)-2\zeta(\bar{2},3)-6\zeta(\bar{5})+\zeta(\bar{3})\zeta(2)-2\zeta(\bar{3},2)\\
	&\quad+3\log^2(2)\zeta(3)+6\log(2)\zeta(\bar{1},3)+6\zeta(\bar{1},\bar{1},3)+3\zeta(2,3)+3\zeta(\bar{2})\zeta(2)\log(2)+3\zeta(\bar{2})\zeta(\bar{1},2)\\
	&\quad+3\log(2)\zeta(\bar{2},2)+12\log\left(2\right)\zeta(\bar{4})+3\zeta(\bar{2},\bar{1},2)+3\zeta(\bar{1},\bar{2},2)+3\zeta(3,2)+3\zeta(\bar{2},\bar{3})\\
	&\quad+9\zeta(\bar{1},\bar{4})+7\zeta(5)-\log^3(2)\zeta(2)-3\log^2(2)\zeta(\bar{1},2)-6\log(2)\zeta(\bar{1},\bar{1},2)-3\log(2)\zeta(2,2)\\
	&\quad-6\log(2)\zeta(\bar{1},\bar{3})-3\zeta(2,\bar{1},2)-3\zeta(\bar{1},2,2)-6\zeta(\bar{1},\bar{1},\bar{3})-3\zeta(2,\bar{3})-3\log(2)\zeta^2(\bar{2})+\log^3(2)\zeta(\bar{2}).
	\end{align*}
\end{exa}

Finally, we conclude by presenting a general statement regarding the parity results for cyclotomic Euler sums of arbitrary order.
\begin{thm}\label{thm-parityc-C-ES-one} Let $x,x_1,\ldots,x_r$ be roots of unity, and $p_1,\ldots,p_r,q\geq 1$ with $(p_j,x_j)\ (j=1,2,\ldots,r) $ and $ (q,xx_1\cdots x_r)\neq (1,1)$. We have
\begin{align*}
&(-1)^r S_{p_1,p_2,\ldots,p_r;q}\Big(x_1,x_2,\ldots,x_r;(xx_1\cdots x_r)^{-1}\Big)
\\&\quad+(-1)^{p_1+\cdots+p_r+q}S_{p_1,p_2,\ldots,p_r;q}\Big(x_1^{-1},x_2^{-1},\ldots,x_r^{-1};xx_1\cdots x_r\Big)
\end{align*}
reduces to a combination of sums of lower orders.
\end{thm}
\begin{proof}
To prove this general theorem, it is necessary to consider the residue computations of the following contour integral:
\begin{align*}
\oint\limits_{\left( \infty  \right)}F_{p_1p_2\cdots p_r,q}(s)ds:=\oint\limits_{\left( \infty  \right)} \frac{\Phi(s;x)\phi^{(p_1-1)}(s;x_1)\cdots\phi^{(p_r-1)}(s;x_r)}{(p_1-1)!\cdots (p_r-1)!s^q}(-1)^{p_1+\cdots+p_r-r} ds=0.
\end{align*}
Clearly, all integers are poles of the integrand in the entire complex plane, with $s=n\ (n\in\N)$ being simple poles, $s=-n\ (n\in \N)$ being poles of order $p_1+\cdots+p_r+1$, and $s=0$ being a pole of order $p_1+\cdots+p_r+q+1$. Applying Lemma \ref{lem-redisue-thm}, we have
\begin{align}\label{residue-sums-CESAO}
\sum_{n=1}^\infty \left(\Res\left(F_{p_1p_2\cdots p_r,q}(s),n\right)+\Res\left(F_{p_1p_2\cdots p_r,q}(s),-n\right)\right)+\Res\left(F_{p_1p_2\cdots p_r,q}(s),0\right)=0.
\end{align}
Applying \eqref{Lexp-phi--n-diffp-1}-\eqref{LEPhi-function} to compute the residues and then substituting these residue values into \eqref{residue-sums-CESAO} yields
\begin{align*}
&\sum_{n=1}^\infty \frac{\zeta_{n-1}(p_1;x_1)\zeta_{n-1}(p_2;x_2)\cdots \zeta_{n-1}(p_r;x_r)}{n^q}(xx_1\cdots x_r)^{-n}\\
&+(-1)^{p_1+\cdots+p_r+q}\sum_{n=1}^\infty \frac{\zeta_{n}(p_1;x^{-1}_1)\zeta_{n}(p_2;x^{-1}_2)\cdots \zeta_{n}(p_r;x^{-1}_r)}{n^q}(xx_1\cdots x_r)^{n}\\
&\quad+\{\text{combinations of lower-order sums}\}=0.
\end{align*}
Using the identity $\zeta_{n-1}(p;x) = \zeta_n(p;x) - x^n/n^p$ and the definition of cyclotomic Euler sums, the theorem can be proved through elementary calculations.
\end{proof}

Similarly, by considering the residue computation of the following general contour integral, one can obtain more results analogous to Theorems \ref{thm-double-CES-Comb} and \ref{thm-parityc-triple-ES-one}:
\begin{align*}
\oint\limits_{\left( \infty  \right)}G_{p_1p_2\cdots p_r,q}(s)ds:=\oint\limits_{\left( \infty  \right)} \frac{\phi^{(p_1-1)}(s;x_1)\cdots\phi^{(p_r-1)}(s;x_r)}{(p_1-1)!\cdots (p_r-1)!s^q}(-1)^{p_1+\cdots+p_r-r} ds=0.
\end{align*}
This paper will not undertake the calculation, but interested readers may attempt to do so.
\begin{re} Theorem \ref{thm-parityc-C-ESmain} is obtained by replacing $x$ with $(xx_1\cdots x_r)^{-1}$ in Theorem \ref{thm-parityc-C-ES-one}.
Indeed, the method of contour integration can fully provide explicit formulas for cyclotomic Euler sums of arbitrary order. However, due to the complexity of the formulas, no attempt was made to compute and present an explicit formula. All examples presented in this paper have been numerically verified for correctness using \emph{Mathematica}.
\end{re}

\begin{re}
The method presented in this paper can also be applied to investigate many other problems related to Dirichlet series of various forms. For instance, by considering the contour integral
\begin{align*}
\oint\limits_{\left( \infty  \right)} \frac{\Phi(s;x)\phi^{(p_1-1)}(s+a_1;x_1)\cdots\phi^{(p_r-1)}(s+a_r;x_r)}{(p_1-1)!\cdots (p_r-1)!(s+a)^q}(-1)^{p_1+\cdots+p_r-r} ds=0\quad(a,\forall a_j\notin \Z)
\end{align*}
under more general settings, it can be used to study the parity properties of multiple Hurwitz polylogarithms. This approach thereby encompasses the parity analysis of objects such as Hoffman's multiple $t$-values, Kaneko-Tsumura's multiple $T$-values, and their cyclotomic analogues.
\end{re}

\medskip

{\bf Declaration of competing interest.}
The authors declares that they has no known competing financial interests or personal relationships that could have
appeared to influence the work reported in this paper.

{\bf Data availability.}
No data was used for the research described in the article.

{\bf Acknowledgments.} Ce Xu gratefully acknowledges the invitation from Professor Chengming Bai of Nankai University to the Chern Institute of Mathematics and from Professor Shaoyun Yi of Xiamen University to the Tianyuan Mathematical Center in Southeast China (TMSE). This work commenced during these visits.

\end{document}